\documentclass[11pt]{article}

\usepackage[utf8]{inputenc}
\usepackage[english]{babel}
\usepackage{amsmath,amssymb,amsfonts,amsthm}
\usepackage{multirow}
\usepackage{xcolor}
\usepackage{hyperref}
\hypersetup{colorlinks=true,linkcolor=blue,filecolor=blue,citecolor = blue,urlcolor=blue}

\usepackage{subfig} 

\usepackage{array}

\usepackage{listings}

\usepackage{graphicx}
\usepackage[scale=0.8]{geometry}

\usepackage{subfig}

\theoremstyle{plain}
\newtheorem{theorem}{Theorem}
\newtheorem{lemma}[theorem]{Lemma}
\newtheorem{proposition}[theorem]{Proposition}
\newtheorem{corollary}[theorem]{Corollary}

\theoremstyle{definition}
\theoremstyle{remark}

\newtheorem{remark}{Remark}

\newcommand{\prob}[1]{\mathbb{P}\left(#1\right)}

\newcommand{\eqd}{\stackrel{(d)}{=}}
\newcommand{\ind}[1]{{\bf 1}_{#1}}
\newcommand{\di}{\mathrm{d}}
\renewcommand{\tilde}[1]{\widetilde{#1}}

\newcommand{\NN}{\mathbb{N}}
\newcommand{\RR}{\mathbb{R}}

\newcommand{\para}{\mathfrak{p}}
\newcommand{\RN}{Radon--Nikodym}
\newcommand{\card}[1]{\text{card}{\left( #1 \right)}}
\newcommand{\Bullet}{bullet }
\newcommand{\BULLET}{Bullet }
\newcommand{\Loop}{loop  }

\newcommand{\red}[1]{{\color{red} #1}}
\newcommand{\blue}[1]{{\color{blue} #1}}
\newcommand{\grey}[1]{{\color{lightgray} #1}}

\let\originalleft\left
\let\originalright\right
\renewcommand{\left}{\mathopen{}\mathclose\bgroup\originalleft}
\renewcommand{\right}{\aftergroup\egroup\originalright}

\newcommand{\rb}{-1mm}
\newcommand{\VE}{\raisebox{\rb}{\includegraphics{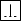}}} 
\newcommand{\VS}{\raisebox{\rb}{\includegraphics{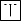}}} 
\newcommand{\VB}{\raisebox{\rb}{\includegraphics{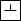}}} 
\newcommand{\VA}{\raisebox{\rb}{\includegraphics{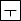}}} 
\newcommand{\VT}{\raisebox{\rb}{\includegraphics{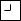}}} 
\newcommand{\HE}{\raisebox{\rb}{\includegraphics{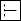}}} 
\newcommand{\HS}{\raisebox{\rb}{\includegraphics{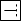}}} 
\newcommand{\HB}{\raisebox{\rb}{\includegraphics{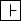}}} 
\newcommand{\HA}{\raisebox{\rb}{\includegraphics{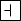}}} 
\newcommand{\HT}{\raisebox{\rb}{\includegraphics{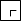}}} 
\newcommand{\OB}{\raisebox{\rb}{\includegraphics{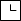}}} 
\newcommand{\OA}{\raisebox{\rb}{\includegraphics{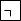}}} 
\newcommand{\CC}{\raisebox{\rb}{\includegraphics{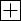}}} 

\newcommand{\rbt}{-0.7mm}
\newcommand{\sct}{0.7}
\newcommand{\tVE}{\raisebox{\rbt}{\includegraphics[scale=\sct]{Icones/VE.pdf}}} 
\newcommand{\tVS}{\raisebox{\rbt}{\includegraphics[scale=\sct]{Icones/VS.pdf}}} 
\newcommand{\tVB}{\raisebox{\rbt}{\includegraphics[scale=\sct]{Icones/VB.pdf}}} 
\newcommand{\tVA}{\raisebox{\rbt}{\includegraphics[scale=\sct]{Icones/VA.pdf}}} 
\newcommand{\tVT}{\raisebox{\rbt}{\includegraphics[scale=\sct]{Icones/VT.pdf}}} 
\newcommand{\tHE}{\raisebox{\rbt}{\includegraphics[scale=\sct]{Icones/HE.pdf}}} 
\newcommand{\tHS}{\raisebox{\rbt}{\includegraphics[scale=\sct]{Icones/HS.pdf}}} 
\newcommand{\tHB}{\raisebox{\rbt}{\includegraphics[scale=\sct]{Icones/HB.pdf}}} 
\newcommand{\tHA}{\raisebox{\rbt}{\includegraphics[scale=\sct]{Icones/HA.pdf}}} 
\newcommand{\tHT}{\raisebox{\rbt}{\includegraphics[scale=\sct]{Icones/HT.pdf}}} 
\newcommand{\tOB}{\raisebox{\rbt}{\includegraphics[scale=\sct]{Icones/OB.pdf}}} 
\newcommand{\tOA}{\raisebox{\rbt}{\includegraphics[scale=\sct]{Icones/OA.pdf}}} 
\newcommand{\tCC}{\raisebox{\rbt}{\includegraphics[scale=\sct]{Icones/CC.pdf}}} 

\newcommand{\TABM}{\text{CBMC}}
\newcommand{\PPP}{\text{PPP}}

\newcommand{\sk}{\stackrel{(\text{sk})}{\sim}}
\newcommand{\set}{\mathcal{S}}

\title{Quasi-reversible \Bullet models:\\colliding \Bullet model with creations and\\ a new(?)\ \Loop model}
\author{Jérôme Casse\footnote{jerome.casse.math@gmail.com}\\Université Paris-Saclay, CNRS\\Laboratoire de mathématiques d’Orsay\\91405 Orsay, France}
\date{\today}

\begin{document}

\maketitle

\begin{abstract}
  We consider a large class of \Bullet models that contains, in particular, the colliding \Bullet model with creations and a new \Loop model. For this large class of \Bullet models, we give sufficient conditions on their parameter to be $\text{rot}(\pi)$-quasi-reversible and to be $\text{rot}(\pi/2)$-quasi-reversible. Moreover, those conditions assure them that one of their stationary measures is described by a Poisson point process. These results, applied to the colliding \Bullet model with creations, are the first steps to study its non-empty stationary measure, and, applied to the \Loop model, prove its invariance according to all the symmetries of the square.
\end{abstract}

\section{Introduction}  \label{sec:intro}
A notorious and difficult bullets problem, and many of its variants, have been studied in the last decade. Send consecutively an infinite number of bullets. If all the bullets have the same speed, they never collide, and there is no problem; but what happens if the speeds of bullets are randomly distributed? This problem is known under at least two names: colliding bullets problem and total annihilation ballistic problem.\par
It has been studied in different contexts and for different distributions of speeds in many recent papers. Let us mention the works when the speeds are uniformly distributed on an interval $[a,b]$~\cite{BM20}, when the speeds are distributed on a set of size two~\cite{BL21} or three~\cite{HST21}, and in many other contexts~\cite{ABL20,ADJLPR24}. Interested readers on this subject should also refer to references therein.\par
\medskip
In~\cite{BCGKP15}, a model in an equilibrium context has been studied, called Colliding \BULLET Model with Creations (CBMC) here. In the model, there are two types of bullets, those of speed $-1$ and those of speed $1$. To obtain an equilibrium model, they add the following rule to create bullets: bullets with speed $-1$ (resp.\ $1$) create bullets of speed $1$ (resp.\ $-1$) at rate $1$. On Figure~\ref{fig:CBMC}, there is a realisation of the dynamic. For this model, they show that it has only two stationary measures invariant by translation: the empty one and another one with an infinite number of bullets. Nevertheless, this second one is hard to describe and does not have a simple form.\par
In addition, in the same work, they also study \Bullet models such that, when two bullets collide, they are not always both annihilated, but:
\begin{enumerate}
\item with probability $p_1$, the bullet of speed $1$ survives and the one of speed $-1$ disappears;
\item with probability $p_{-1}$, the bullet of speed $-1$ survives and the one of speed $1$ disappears;
\item with probability $p_{0}$, both bullets disappear;
\item with probability $1-p_1-p_{-1}-p_{0}$, they both survive and continue their trajectories.
\end{enumerate}
In some particular cases, e.g.\ when $p_1 = p_{-1} = 1/2$, they show that one of the stationary measures is two independent Poisson Point Processes with both intensities $1$: one for the distribution of bullets with speed~$-1$ and one for the distribution of bullets with speed~$1$.\par

\begin{figure}
  \begin{center}
    \includegraphics{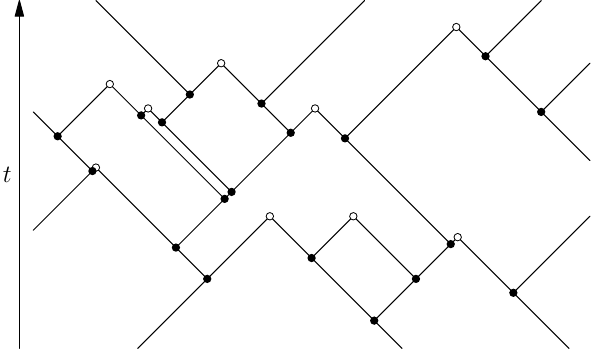}
  \end{center}
  \caption{A realisation of the colliding \Bullet model with creations. Black dots correspond to creations and white dots to annihilations.}%
  \label{fig:CBMC}
\end{figure}

A crucial motivation for the redaction of the present article is the description of the non-empty stationary measure of the CBMC. In particular, see Proposition~\ref{prop:TABM}, we give a simple description of the non-empty stationary measure of the CBMC. Moreover, see Proposition~\ref{prop:CBMCqr}, we prove that its space-time diagram has the same distribution as the space-time diagram, rotated by an angle $\pi/2$, of another \Bullet model whose stationary measure is two independent Poisson Point Process with both intensities $1$.\par
\medskip
In addition, even in older literature, some important very well-known models can be seen as \Bullet models. For example, the Hammersley's lines process, introduced to study the longest sub-increasing sequence of a permutation, could be seen as a colliding \Bullet model with random sources uniformly distributed on the plane~\cite{AD95,CG05}, see Figure~\ref{fig:HL}. Let us also mention the variant of Hammersley's lines process studied in~\cite{BGGS18} that can be seen, in the geometric case, as a \Bullet model.\par
\medskip
Moreover, if, in a Hammersley's lines process, the bullets are allowed to turn (their speed goes from~$-1$ to $1$, or the reverse) at rate $1$, a model with avoiding and self-avoiding loops emerges. Under (one of) its stationary measure, this model is invariant according to the eight symmetries of the square, see Proposition~\ref{prop:loop}. To the best knowledge of the author, this \Loop model seems to be new and not directly related to the most known \Loop models (Ising model, CLE : conformal loop ensemble, $O(n)$-loop model on random planar maps, etc.). For now, the only way to define it is only threw a dynamical system on $\RR$ and not directly on the space $\RR^2$. A realisation of this model is represented on Figure~\ref{fig:LM}. 

\begin{figure}
  \begin{center}
    \includegraphics{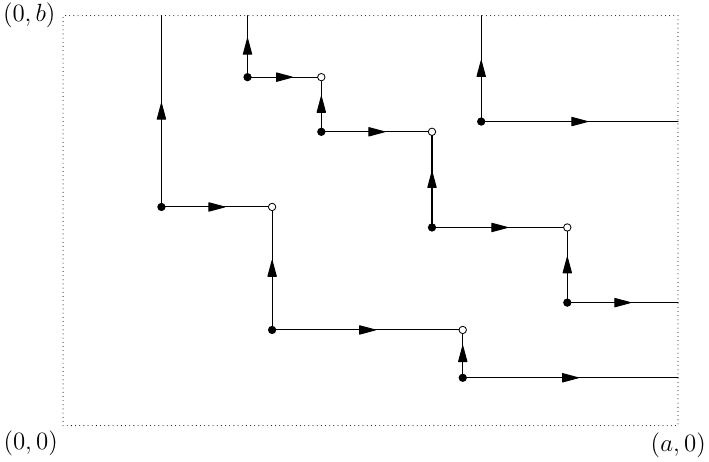}
  \end{center}
  \caption{A realisation of the Hammersley's lines process inside the rectangle $[0,a] \times [0,b]$. Black dots correspond to sources and white dots to annihilations.}%
  \label{fig:HL}
\end{figure}

\begin{figure}
  \begin{center}
    \includegraphics{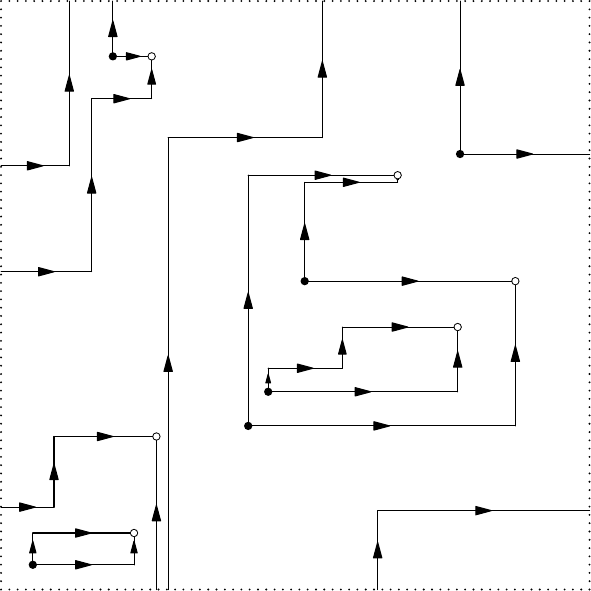}
  \end{center}
  \caption{An example of the new(?)\ \Loop model inside the square $[0,5] \times [0,5]$. Black dots correspond to sources and white dots to annihilations.}%
  \label{fig:LM}
\end{figure}

\paragraph{Content:}
In Section~\ref{sec:bm}, we define a class of \Bullet models indexed by $8$ parameters with only two possible speeds for particles. In Section~\ref{sec:ex}, we show that some examples given in the Introduction are \Bullet models. In Sections~\ref{sec:ic},~\ref{sec:st} and~\ref{sec:im}, we define formally the space-time diagram of a \Bullet model as well as the notion of stationarity in that context.\par
In Section~\ref{sec:qr}, we define the notion of quasi-reversibility and we state, among others, the two main theorems of the article, Theorems~\ref{thm:pi} and~\ref{thm:pi2}, and two important corollaries, Corollaries~\ref{cor:pi} and~\ref{cor:pi2}. Theorem~\ref{thm:pi} gives sufficient conditions for two \Bullet models to be $\text{rot}(\pi)$-quasi-reversible and Theorem~\ref{thm:pi2} does it for $\text{rot}(\pi/2)$-quasi-reversibility. In Section~\ref{sec:appli}, we apply them to the Colliding \BULLET Model with Creations and to the new \Loop model proving Propositions~\ref{prop:CBMCqr} and~\ref{prop:loop}.\par
Finally, in Section~\ref{sec:heu}, heuristics of both Theorems~\ref{thm:pi} and~\ref{thm:pi2} are given, and, in Section~\ref{sec:proof}, their proofs are done as well as the ones of Corollaries~\ref{cor:pi} and~\ref{cor:pi2}.

\section{Definition and presentation of the results}\label{sec:def}%
\subsection{\BULLET models}\label{sec:bm}%
The \Bullet models, considered in the article, are parameterised by the $8$ following parameters:
\begin{itemize}
\item the \emph{spontaneous creation rate} $\lambda_0 \in \RR_{\geq 0}$,
\item the \emph{vertical and horizontal split rates} $\lambda_V \in \RR_{\geq 0}$ and $\lambda_H \in \RR_{\geq 0}$,
\item the \emph{vertical and horizontal turn rate} $\tau_V \in \RR_{\geq 0}$ and $\tau_H \in \RR_{\geq 0}$,
\item the \emph{annihilation probability} $p_0 \in [0,1]$, and
\item  the \emph{vertical and horizontal coalescence probability} $p_V \in [0,1]$ and $p_H \in [0,1]$
\end{itemize}
such that $p_0+p_V+p_H \leq 1$. The vocabulary comes from the one used in~\cite{BCES23} about Poisson-Kirchhoff Systems.\par

From this parameter $\para = (\lambda_0,\lambda_V,\lambda_H, \tau_V, \tau_H, p_V,p_H, p_0)$ and an initial condition given on the $x$-and $y$-axes, we define a random system of horizontal and vertical lines inside the quarter plane $[0,\infty)^2$.\par

The \emph{initial condition} is a couple $(\mathcal{C}_X,\mathcal{C}_Y)$ where $\mathcal{C}_X$, resp.~$\mathcal{C}_Y$, is a (finite or enumerable) set of points on the non-negative $x$-axis, resp.\ on the non-negative $y$-axis. Hence, an element of $\mathcal{C}_X$ is a point $(x,0) \in \RR_{\geq 0} \times \{0\}$. Similarly, an element of $\mathcal{C}_Y$ is a point $(0,y) \in \{0\} \times \RR_{\geq 0}$. Moreover, we assume that both of them are locally finite, i.e.\ there is no accumulation point.\par

Finally, take also a PPP $\Xi_0$, called ex-nihilo creation points, on $(0,\infty)^2$ with intensity $\lambda_0 \di x \di y$. From the initial condition $(\mathcal{C}_X,\mathcal{C}_Y)$ and the parameter $\para = (\lambda_0,\lambda_V,\lambda_H, \tau_V, \tau_H, p_V,p_H, p_0)$, the \Bullet model is constructed with the following rules:
\begin{enumerate}
\item[$1_V$.] for each point $(x,0) \in \mathcal{C}_X$, a vertical line from the point $(x,0)$ goes up;
\item[$1_H$.] for each point $(0,y) \in \mathcal{C}_Y$, a horizontal line from the point $(0,y)$ goes right;
\item[$1_0$.] for each point $(x,y) \in \Xi_0$, a vertical line going up and a horizontal line going right are started from the point $(x,y)$.
\end{enumerate}
Along a vertical line, two events can occur:
\begin{enumerate}
\item[$2_V(a)$.] at rate $\lambda_V$, a horizontal line starts going right and the vertical one continues going up;
\item[$2_V(b)$.] at rate $\tau_V$, the vertical line turns right and becomes then a horizontal line.
\end{enumerate}
Similarly, along a horizontal line:
\begin{enumerate}
\item[$2_H(a)$.] at rate $\lambda_H$, a vertical line starts going up and the horizontal one continues going right;
\item[$2_H(b)$.] at rate $\tau_H$, the horizontal line turns up and becomes then a vertical line.
\end{enumerate}
Finally, when two lines (one vertical and one horizontal) intersect, four events can occur:
\begin{enumerate}
\item[3(a).] with probability $p_V$, the horizontal line stops and the vertical line continues going up;
\item[3(b).] with probability $p_H$, the vertical line stops and the horizontal line continues going right;
\item[3(c).] with probability $p_0$, both lines stop;
\item[3(d).] with probability $1-(p_V+p_H+p_0)$, both lines continue.
\end{enumerate}
Such construction gives then a random set $\mathcal{L}(\para,(\mathcal{C}_X,\mathcal{C}_Y))$ of segments and lines in the quarter-plane $[0,\infty)^2$.\par
\bigskip
\paragraph{Are any \Bullet models well defined?}
One could ask if this construction above is well defined, in the sense that there is no accumulation point on the whole quarter plane, that is equivalent to have a finite number of lines in any bounded regions.\par
\begin{proposition}
  Assume that the sets $\mathcal{C}_X$ and $\mathcal{C}_Y$ are locally finite a.s., 
  then, for any parameter $\para$, the random set $\mathcal{L}(\para,(\mathcal{C}_X,\mathcal{C}_Y))$ has no accumulation point a.s.\ in the quarter plane $[0,\infty)^2$.  
\end{proposition}

\begin{proof}
  Any \Bullet model can be seen as a Poisson-Kirchhoff system defined in~\cite{BCES23} where the set of intensity is reduced to the singleton $\{0\}$, see~\cite[Section~6.4]{BCES23}. Hence, the Theorem~2~of~\cite{BCES23} applies and, so, any \Bullet model is well-defined.
\end{proof}

\subsection{Examples}\label{sec:ex}%
As the set of speeds considered here is of size two, we cannot model all the examples in the Introduction, in particular, those with more than two possible speeds. Nevertheless, up to some linear deformation of the plane, colliding \Bullet models with two speeds are \Bullet models with parameter $\para = (0,0,0,0,0,0,0,1)$.\par
\medskip

The CBMC defined in~\cite{BCGKP15} is the \Bullet model with parameter $\para_{\TABM} = (0,1,1,0,0,0,0,1)$ after a rotation of angle $\pi/4$ and a dilation of factor $1/\sqrt{2}$, see Figure~\ref{fig:rot}. Its generalisations when both particles do not always disappear are \Bullet models with parameter $\para = (0,1,1,0,0,p_{-1},p_1,p_0)$.\par
\medskip

\begin{figure}
  \begin{center}
    \begin{tabular}{>{\centering\arraybackslash}m{.4\textwidth} >{\centering\arraybackslash}m{.4\textwidth}}
      \includegraphics[width=0.4\textwidth]{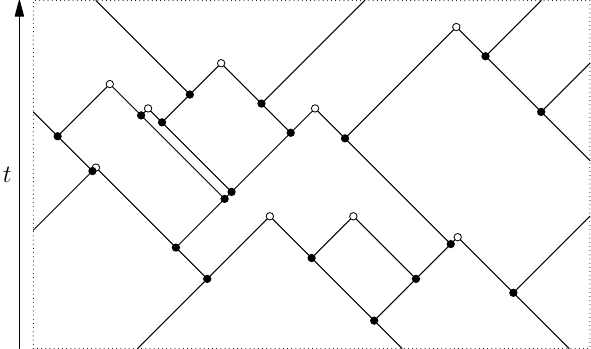} & \includegraphics[width=0.4\textwidth]{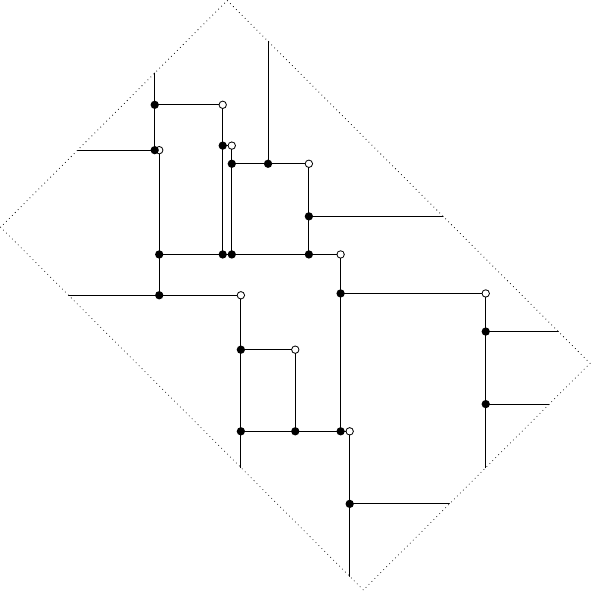}
    \end{tabular}
  \end{center}
  \caption{On the left, a realisation of the colliding \Bullet model with creations. On the right, the same realisation after a rotation of an angle $-\pi/4$ and a dilation.}%
  \label{fig:rot}
\end{figure}

To end with the example in the Introduction, we mention that the Hammersley's lines process corresponds to the \Bullet model with parameter $\para = (1,0,0,0,0,0,0,1)$, and the model of~\cite{BGGS18} in the geometric case to the
\Bullet model with parameter $\para = (1,0,0,0,0,0,1-\alpha,\alpha)$. The new(?)\ \Loop model, Hammersley's lines with turns, corresponds to the parameter $\para_{\text{loop}} = (1,0,0,1,1,0,0,1)$.\par
\medskip

Finally, let us introduce two other \Bullet models with turns. The first one with parameter $\para_V = (0,0,1,1,0,1,0,0)$ corresponds to the following rules:
\begin{itemize}
\item horizontal particles can only create vertical particles at rate $1$,
\item vertical ones can only turn right at rate $1$ and,
\item when an encounter occurs, the horizontal particle stops and the vertical one goes on.
\end{itemize}
The other model with parameter $\para_H = (0,1,0,0,1,0,1,0)$ is the same, replacing vertical by horizontal and vice versa. These two \Bullet models are the keys to understand the CBMC and its stationary measure. Indeed, in Section~\ref{sec:appli}, we deduce from Corollary~\ref{cor:pi}, that these two \Bullet models have a very simple and explicit stationary measure: two independent PPP; and, then, in Proposition~\ref{prop:CBMCqr}, we prove that their space-time diagrams rotated by an angle $-\pi/2$ and, resp.\ $\pi/2$, have the same distribution as the one of the CBMC under its non empty stationary measure.

\medskip

On Figure~\ref{fig:examples}, realisations of some mentioned examples are drawn.
\begin{figure}
  \begin{center}
    \begin{tabular}{ccc}
      \includegraphics[width=7cm]{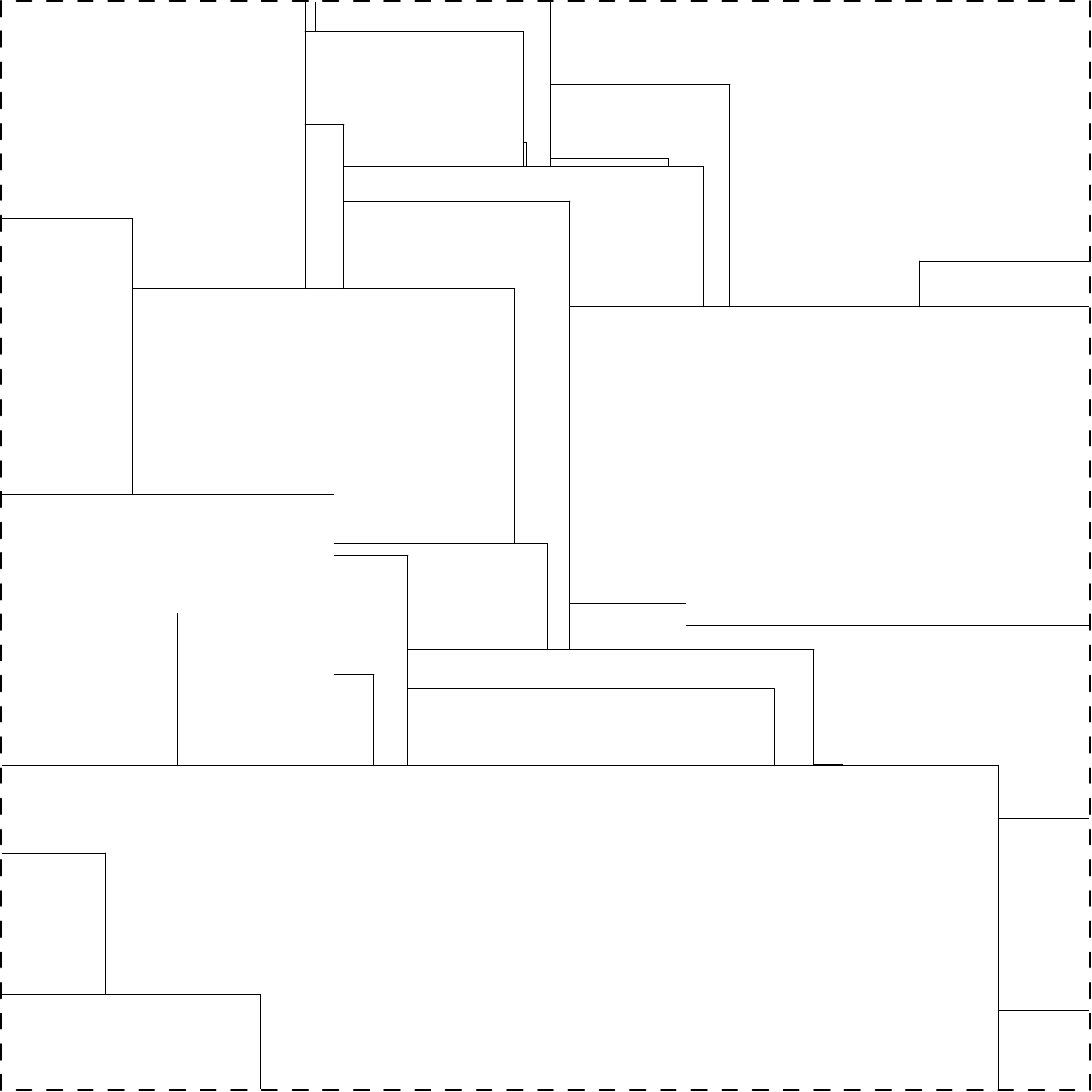} & \hspace{1cm} & \includegraphics[width=7cm]{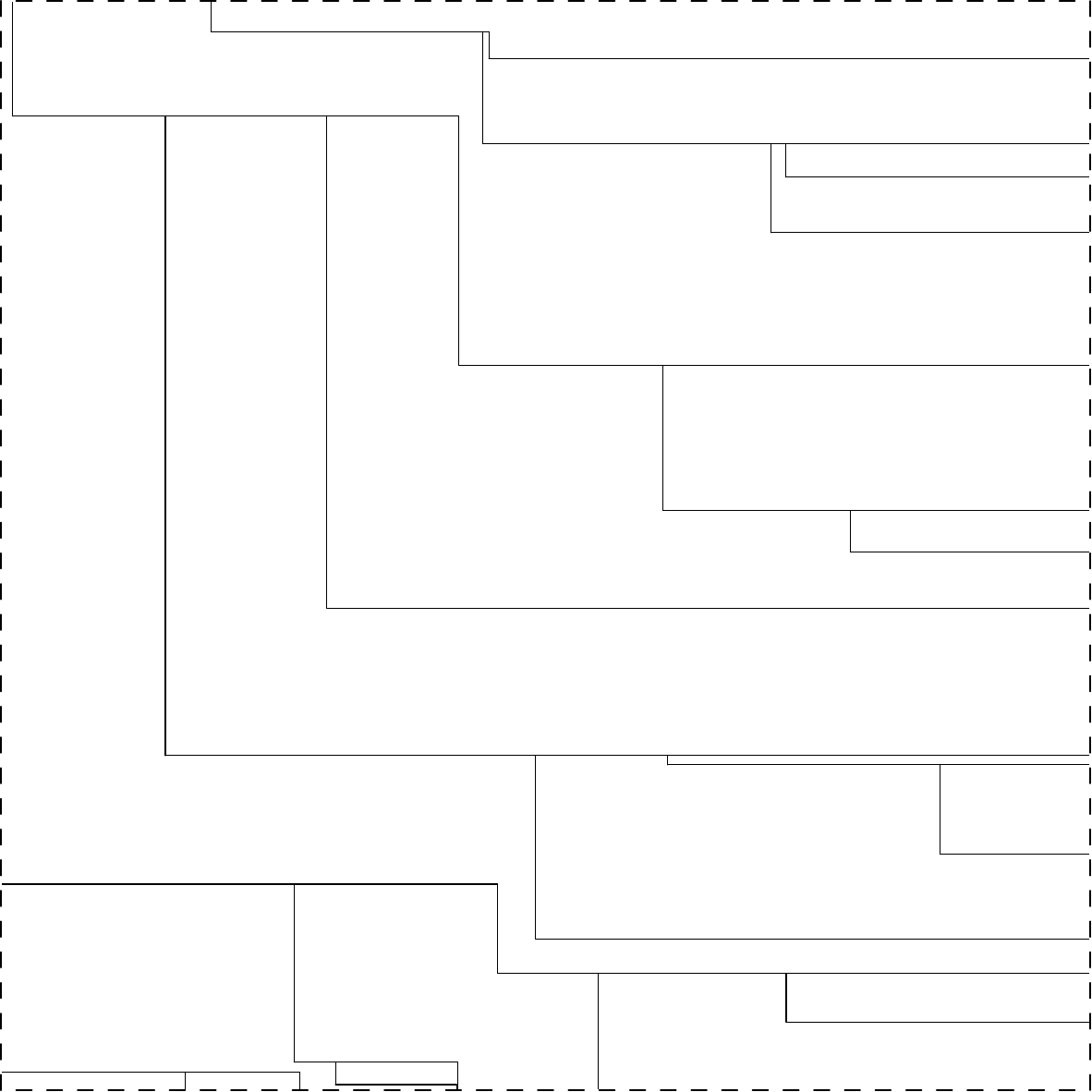}\\
      $\para_{\TABM} = (0,1,1,0,0,0,0,1)$ & & Model of~\cite{BGGS18} $\para = (1,0,0,0,0,0,3/4,1/4)$\\
      \vspace{1cm} & & \\
      \includegraphics[width=7cm]{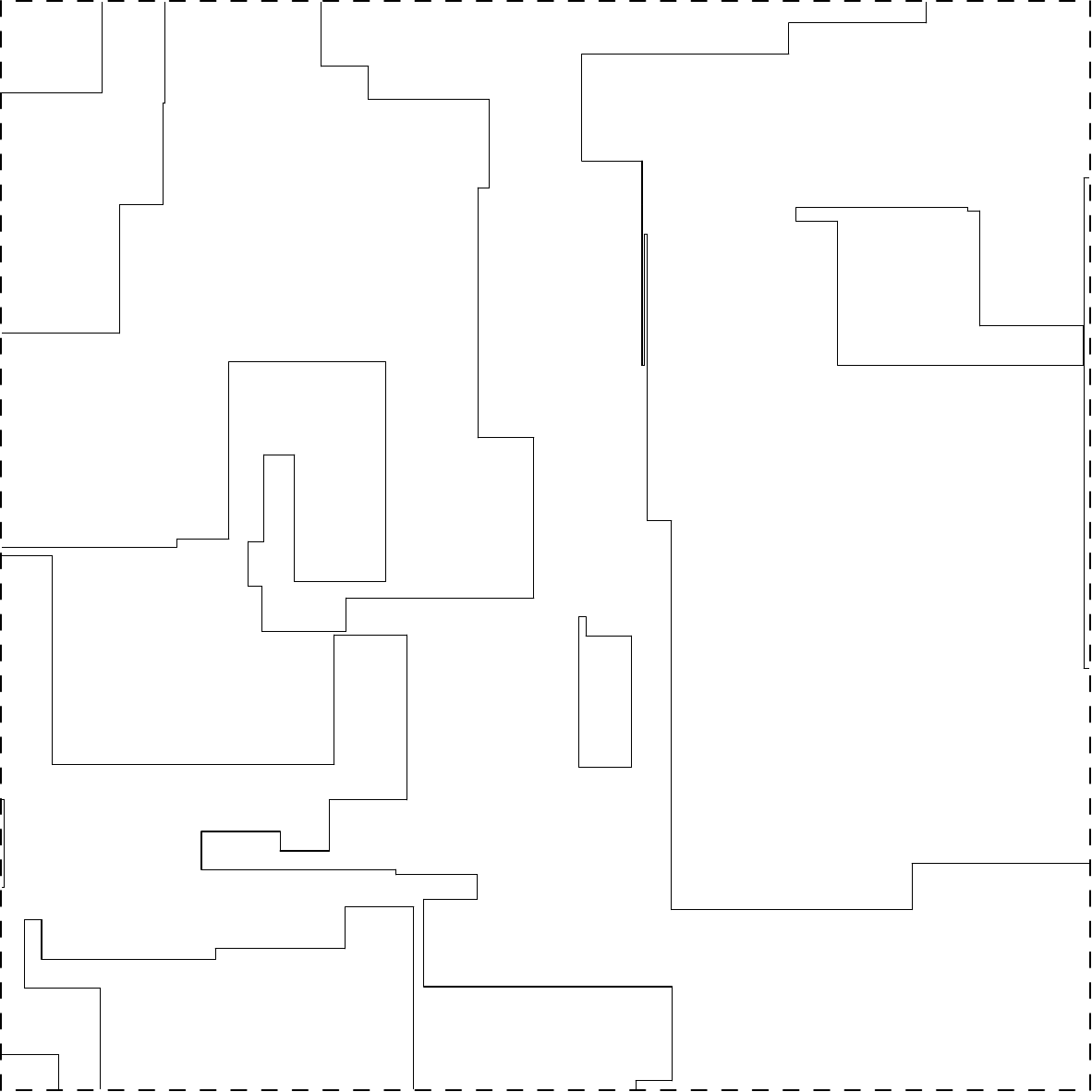} & & \includegraphics[width=7cm]{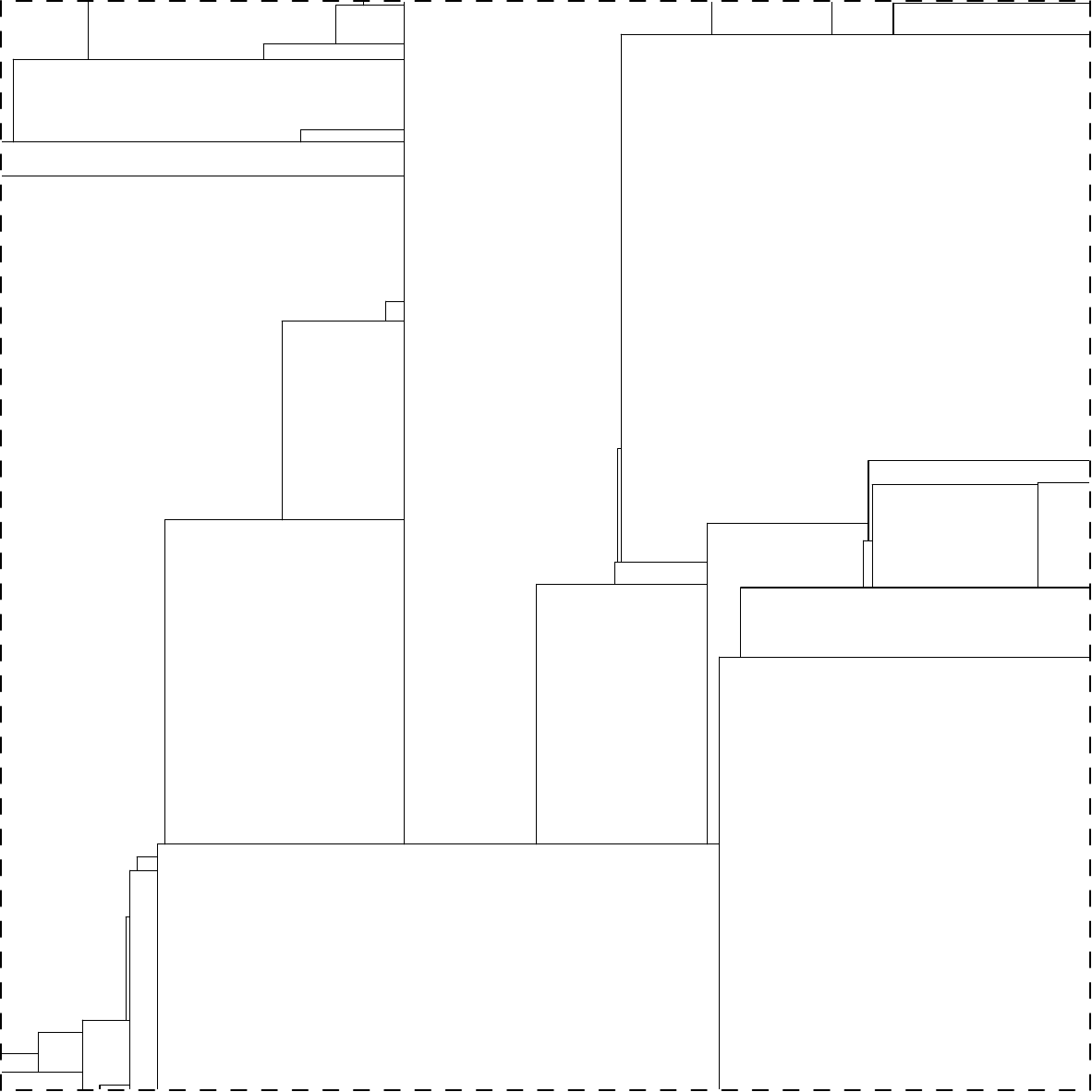}\\
      $\para_{\text{loop}} = (1,0,0,1,1,0,0,0)$ & & $\para_V = (0,0,1,1,0,1,0,0)$ \\
    \end{tabular}
  \end{center}
  \caption{Realisations (obtained by simulations) of \Bullet models on the rectangle $[0,5] \times [0,5]$ where the initial condition $(\mathcal{C}_X,\mathcal{C}_Y)$ are two independent PPPs with both intensities $1$. The code of simulations is given in Annex~\ref{ann:code}.} \label{fig:examples}
\end{figure}


\subsection{Initial conditions}\label{sec:ic}%
Now, we formalise the set of probability laws for the initial condition $(\mathcal{C}_X,\mathcal{C}_Y)$.\par

To consider random initial conditions $(\mathcal{C}_X,\mathcal{C}_Y)$, the law of $(\mathcal{C}_X,\mathcal{C}_Y)$ is represented here as a random point process on the line $\RR$. Basically, a random point process on the line $\RR$ is a probability measure on the set of countable locally finite sets of points. Hence, a realisation of such a process is a set of points $X = \{x_i \in \RR : i \in I\}$ where $I$ is countable and there is no accumulation point in $X$. The law of a random point process could be given by the law of its \emph{counting measure} $N$, that is, the knowledge of
\begin{displaymath}
  (\prob{N([a,b])=i} = \prob{ \card{X \cap [a,b]} = i})_{i \in \NN, a \in \RR,b \in \RR, a \leq b},
\end{displaymath}
see~\cite[Chapter~3.1]{DVJ03} or~\cite[Chapter~1]{Bremaud20} for some reference books on random point processes. The most famous point process is the Poisson Point Process (PPP). The law of the counting measure of a PPP on $\RR$ with intensity $\lambda \in [0,\infty)$ is, for any $i \in \NN$, $a,b \in \RR$, $a \leq b$,
\begin{displaymath}
  \prob{ \card{X \cap [a,b]} = i} = e^{-\lambda(b-a)} \frac{\lambda^i (b-a)^i}{i!}.
\end{displaymath}
\par
Now, let $X = \{x_i \in \RR : i\in I\}$ be a random point process of counting measure $N$. We define the initial condition $(\mathcal{C}_X,\mathcal{C}_Y)$ by
\begin{equation}
  \mathcal{C}_X = \{(x,0) : x \in X \text{ and } x \geq 0\} \text{ and }
  \mathcal{C}_Y = \{(0,-x) : x \in X \text{ and } x < 0\}.
\end{equation}
In the sequel of this article, we do the abuse of language writing that the initial condition $(\mathcal{C}_X,\mathcal{C}_Y)$ is distributed according to the counting measure~$N$.\par
\smallskip
In the following, some important results concern specific stationary measures where $\mathcal{C}_X$ and $\mathcal{C}_Y$ are two independent PPP with intensities $\nu_V$ and $\nu_H$. Hence, let us introduce the notation $N_{\PPP(\nu_H,\nu_V)}$ for their counting measure, i.e.\ for any $a \in \RR$, $b \in \RR$, $a<b$,
\begin{equation}
  \prob{N_{\PPP(\nu_H,\nu_V)}([a,b]) = i} = \begin{cases}
    e^{-\nu_H (b-a)} \frac{\nu_H^i (b-a)^i}{i!} & \text{if } b \leq 0,\\
    e^{-\nu_V (b-a)} \frac{\nu_V^i (b-a)^i}{i!} & \text{if } 0 \leq a,\\
    \sum_{j=0}^i e^{a \nu_H - b \nu_V} \frac{\nu_H^j \nu_V^{i-j} (-a)^j b^{i-j}}{j! (i-j)!} & \text{if } a < 0 < b.
  \end{cases}
\end{equation}

\subsection{Notations for the space-time diagram} \label{sec:st}
Let us introduce some notations about the space-time diagrams of \Bullet models used all along the article.\par
Take a \Bullet model with parameter $\para$ and a counting measure $N$ on $\RR$. We denote by $\mathcal{L}_{[0,\infty)^2}(\para,N)$ the random set of segments and lines in the quarter-plane obtained by applying the \Bullet model with parameter $\para$ to the initial condition $(\mathcal{C}_X,\mathcal{C}_Y) \sim N$. A formal representation of its set of configurations is given in the beginning of Section~\ref{sec:dens}.\par
For any $(a,b) \in \RR$, let us define the set of segments $\mathcal{L}_{[a,\infty) \times [b,\infty)}(\para,N)$ on $\{(x,y) : x \geq a ,y \geq b\}$ obtained by the \Bullet model with parameter $\para$ with the initial condition $(\mathcal{C}_X,\mathcal{C}_Y) \sim N$ where $\mathcal{C}_X$ is the random point process on the line $\{(x+a,b) : x \geq 0\}$ and $\mathcal{C}_Y$ is the random point process on the line $\{(a,y+b) : y \geq 0\}$. It is $\mathcal{L}_{[0,\infty)^2}(\para,N)$ translated by the vector $(a,b)$.\par 
For any $a,b \in \RR$, for any $a_1,b_1,a_2,b_2$ such that $a \leq a_1 < a_2 \leq \infty$ and $b \leq b_1 < b_2 \leq \infty$, we denote by $\mathcal{L}_{[a,\infty)\times[b,\infty)}(\para,N)|_{[a_1,a_2]\times[b_1,b_2]}$ the restriction of $\mathcal{L}_{[a,\infty)\times[b,\infty)}(\para,N)$ to the rectangle $[a_1,a_2] \times [b_1,b_2]$.\par

\subsection{Stationary measures} \label{sec:im}
The counting measure $N$ is said to be \emph{stationary} by the \Bullet model with parameter $\para$ if, for almost any $(a,b) \in \RR_+^2$,
\begin{equation}
  \mathcal{L}_{[0,\infty)^2}(\para,N)|_{[a,\infty)\times[b,\infty)} \eqd \mathcal{L}_{[a,\infty)\times[b,\infty)}(\para,N).
\end{equation}\par
Now, take a \Bullet model with parameter $\para$ under one of its stationary measures $N$. By Kolmogorov's extension theorem, we can define a random set $\mathcal{L}_{\RR^2}(\para,N)$ of segments and lines on the whole plane $\RR^2$. Let denote $\rho(\para,N)$ its probability law.\par

\paragraph{Stationary measure of \Bullet model with parameter $\para_{\TABM}$:}
As mentioned in the Introduction, a motivation of this article is to study the stationary measure of the \Bullet model with parameter $\para_{\TABM}$.\par
With the point of view taken here, its stationary measure is quite simple:
\begin{proposition} \label{prop:TABM}
  The counting measure $N_{\PPP(1,1)}$ is stationary by the \Bullet model with parameter $\para_{\TABM}$.  
\end{proposition}
But, the stationary measure defined in~\cite{BCGKP15} corresponds to the law of an anti-diagonal line in our case and this one is still complicated to study. Nevertheless, due to Proposition~\ref{prop:TABM}, it begins to be doable.\par
Proof of Proposition~\ref{prop:TABM} is deduced from Proposition~\ref{prop:CBMCqr}, see Section~\ref{sec:appli}.

\section{Quasi-reversibility of \Bullet models} \label{sec:qr}
Consider $D_4$ the dihedral group, the group of the symmetries of the square $[-1,1]^2$. The group $D_4$ has $8$ elements: $4$ rotations of centre $(0,0)$ and angle $k\pi/2$ with $k \in \{0,1,2,3\}$, we denote them $s_{k\pi/2}$'s, and $4$ reflections around the diagonal $\{(x,x) : x \in \RR\}$ (we denote it $r$), the $x$-axis $\{(x,0) : x \in \RR\}$ (that is $s_{3\pi/2} \circ r = r \circ s_{\pi/2}$), the $y$-axis $\{(0,x) : x \in \RR\}$ (that is $s_{\pi/2} \circ r = r \circ s_{3\pi/2}$),  and the anti-diagonal $\{(x,-x) : x \in \RR\}$ (that is $s_{\pi} \circ r = r \circ s_{\pi}$). Hence, any element of $D_4$ is also an automorphism of $\RR^2$. For any $g \in D_4$ and any point $(x,y) \in \RR^2$, we denote by $g(x,y)$ its image by $g$. This is illustrated on Figure~\ref{fig:dihedral}. And, so, for any subset $A$ of $\RR^2$, we define the subset $g(A) = \{g(x,y) : (x,y) \in A\}$.\par
\medskip
\begin{figure}
  \begin{center}
    \includegraphics{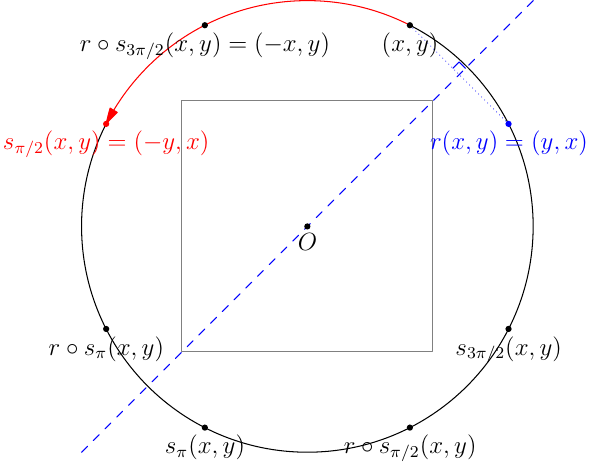}
  \end{center}
  \caption{A point and its eight images by the elements of $D_4$.}%
  \label{fig:dihedral}
\end{figure}

Let $g \in D_4$ be a symmetry of the square. A \Bullet model with parameter $\para$ under one of its stationary measures $N$ is said to be \emph{$g$-quasi-reversible} if there exists a \Bullet model with parameter $\tilde{\para}$ under one of its stationary measures $\tilde{N}$ such that the set $\mathcal{L}_{\RR^2}(\tilde{\para},\tilde{N}) \eqd g(\mathcal{L}_{\RR^2}(\para,N))$. In that case, the couple $(\tilde{\para},\tilde{N})$ is called the \emph{$g$-reverse} of $(\para,N)$. When, in addition, the couples are equal $(\tilde{\para},\tilde{N}) = (\para,N)$, we simply say that $(\para,N)$ is \emph{$g$-reversible}.\par
\medskip
The notion of $g$-quasi-reversibility for any $g \in D_4$ has been introduced first in~\cite{CM20} to study PCA with memory two.\par

\subsection{$r$-quasi-reversibility} \label{sec:r}
First, let us remark that any \Bullet model under any of its stationary measure is $r$-quasi-reversible.
\begin{proposition}[The $r$-quasi-reversibility] \label{prop:r}
  Consider any \Bullet model with parameter $\para$ under one of its stationary measures $N$. Then it is $r$-quasi-reversible and its $r$-reverse $(\tilde{\para},\tilde{N})$ is
  \begin{enumerate}
  \item for any $a,b \in \RR$, $\tilde{N}([a,b]) = N([-b,-a])$,
  \item $\tilde{\lambda}_0 = \lambda_0$, $\tilde{\lambda}_V = \lambda_H$ and $\tilde{\lambda}_H = \lambda_V$,
  \item $\tilde{\tau}_V = \tau_H$ and $\tilde{\tau}_H = \tau_V$,
  \item $\tilde{p}_V = p_H$, $\tilde{p}_H = p_V$ and $\tilde{p}_0 = p_0$.
  \end{enumerate}
\end{proposition}

\begin{proof}
  Reversing the $x$-and $y$-axes, the \Bullet model is a \Bullet model in which the vertical and horizontal particles swap their roles.
\end{proof}

\subsection{$s_\pi$-quasi-reversibility} \label{sec:pi}
Now we state Theorem~\ref{thm:pi} that gives sufficient conditions for two \Bullet models with parameter $\para$ and $\tilde{\para}$ to be $s_\pi$-quasi-reversible. In that case, both it and its reverse get a stationary measure that is two independent PPP.\par
We also state one of its corollaries, Corollary~\ref{cor:pi}, that gives sufficient conditions for a \Bullet model with parameter $\para$ to be $s_\pi$-quasi-reversible. In that case, we can even compute from $\para$ both one of its stationary measures and the parameter $\tilde{\para}$ of its $s_\pi$-reverse.\par

\begin{theorem} \label{thm:pi}
  Consider two \Bullet models, one with parameter $\para$ and the other one with parameter $\tilde{\para}$. Suppose that there exist $\nu_V,\nu_H \in \RR_{\geq 0}$ such that the five following conditions are satisfied:
  \begin{enumerate}
  \item $\tilde{\lambda}_H + \tilde{\tau}_H = \lambda_H + \tau_H$ and $\tilde{\lambda}_V + \tilde{\tau}_V = \lambda_V + \tau_V$,
  \item $\tilde{\lambda}_0 = \lambda_0 = \nu_V \nu_H p_0$, $\tilde{\lambda}_V = \nu_H p_V$ and $\tilde{\lambda}_H = \nu_V p_H$,
  \item $\nu_V \tilde{\tau}_V = \nu_H \tau_H$ and $\nu_H \tilde{\tau}_H = \nu_V \tau_V$,
  \item $\nu_H \tilde{p}_V = \lambda_V$, $\nu_V \tilde{p}_H = \lambda_H$ and $\tilde{p}_0 = p_0$,
  \item $\tilde{p}_V + \tilde{p}_H + \tilde{p}_0 = p_V + p_H + p_0$.
  \end{enumerate}
  
  Then, the measure $N_{\PPP(\nu_H,\nu_V)}$ is stationary for both of them and the couples $(\para,N_{\PPP(\nu_H,\nu_V)})$ and $(\tilde{\para},N_{\PPP(\nu_H,\nu_V)})$ are $s_\pi$-reverse. 
\end{theorem}  

Combining Proposition~\ref{prop:r} with Theorem~\ref{thm:pi} permits to deduce sufficient conditions to get $(s_{\pi} \circ r)$-quasi-reversibility.\par
\medskip
\begin{remark}
  In~\cite{BCES23}, Theorem~3.5 applied to PKS with intensity $0$ permits to obtain sufficient conditions for a \Bullet model to be $s_\pi$-reversible. If we restrict Theorem~\ref{thm:pi} to $s_\pi$-reversibility (that is $\para = \tilde{\para}$), we find back the same conditions.
\end{remark}
\medskip
Now let us state a corollary that involves only one \Bullet model with parameter $\para$. To get a simpler expression of it, let us introduce the three following non-negative quantities depending on $\para$:
\begin{align}
  & A = (\lambda_H+\tau_H)(\lambda_V+\tau_V) - \tau_V \tau_H, \label{eq:A}\\
  & B_V = (p_H+p_V) (\lambda_V + \tau_V) - p_V \lambda_V, \text{ and} \label{eq:BV}\\
  & B_H = (p_H+p_V) (\lambda_H + \tau_H) - p_H \lambda_H. \label{eq:BH}
\end{align}

\begin{corollary} \label{cor:pi}
  Consider a \Bullet model with parameter $\para$ such that $A,B_V,B_H \neq 0$ and
  \begin{equation}
    B_H B_V \lambda_0 = A^2 p_0. \label{eq:pi-rev}
  \end{equation}
  Then, the measure $N_{\PPP(\nu_H,\nu_V)}$ is stationary for it with
  \begin{equation}
    \nu_H = \frac{A}{B_H} \text{ and } \nu_V = \frac{A}{B_V}. \label{eq:nu}
  \end{equation}
  
  Moreover, the couple $(\para,N_{\PPP(\nu_H,\nu_V)})$ is $s_{\pi}$-quasi-reversible and its $s_\pi$-reverse is $(\tilde{\para},N_{\PPP(\nu_H,\nu_V)})$ where the parameter $\tilde{\para}$ is
  \begin{itemize}
  \item $\tilde{\lambda}_0 = \lambda_0$, $\tilde{\lambda}_V = \frac{A}{B_H} p_V$ and $\tilde{\lambda}_H = \frac{A}{B_V} p_H$,
  \item $\tilde{\tau}_V = \frac{B_V}{B_H} \tau_H$ and $\tilde{\tau}_H = \frac{B_H}{B_V} \tau_V$,
  \item $\tilde{p}_V = \frac{B_H}{A} \lambda_V$, $\tilde{p}_H = \frac{B_V}{A} \lambda_H$ and $\tilde{p}_0 = p_0$.
  \end{itemize}
\end{corollary}

\begin{remark}
  The good point of the corollary is that we can compute from the parameter $\para$ both one of its stationary measures and the parameter $\tilde{\para}$ of one of its $s_\pi$-reverse, whereas we need to guess them to apply Theorem~\ref{thm:pi}. The bad point is that the conditions $A,B_V,B_H \neq 0$ are restrictive.
\end{remark}


When $\lambda_0 = p_0 = 0$, the Equation~\eqref{eq:pi-rev} always holds. Hence, 
\begin{corollary} \label{cor:pi0}
  Consider a \Bullet model with parameter $\para$ such that $\lambda_0 = p_0 = 0$ and $A,B_V,B_H \neq 0$. Then, the measure $N_{\PPP(\nu_H,\nu_V)}$ is stationary for $\para$ with $\nu_H$ and $\nu_V$ as defined in Equation~\eqref{eq:nu}. Moreover, because $\lambda_0 = 0$, the empty measure $N_{\PPP(0,0)}$ is stationary for $\para$.
\end{corollary}

\subsection{$s_{\pi/2}$-quasi-reversibility} \label{sec:pi2}
As in the previous section, we state first a theorem that gives sufficient conditions for two \Bullet models, one of a parameter $\para$ and the other one with parameter $\tilde{\para}$, such that the one with parameter $\para$ is $s_{\pi/2}$-quasi-reversible and its $s_{\pi/2}$-reverse is the \Bullet model with parameter $\tilde{\para}$. Then, a corollary is given that establishes sufficient conditions only on $\para$ and that permits to compute $\tilde{\para}$ from $\para$.\par

\begin{theorem} \label{thm:pi2}
  Consider two \Bullet models, one with parameter $\para$ and the other one with parameter $\tilde{\para}$. Suppose that there exists $\nu_V \in \RR_{\geq 0}$ such that the five following conditions are satisfied:
  \begin{enumerate}
  \item $\tilde{\lambda}_H + \tilde{\tau}_H = \lambda_V + \tau_V$ and $\tilde{\lambda}_V + \tilde{\tau}_V = \lambda_H + \tau_H$,
  \item $\tilde{\lambda}_0 = \lambda_0 = \nu_V \tau_V$, $\tilde{\lambda}_V = \nu_V p_H$ and $\tilde{\lambda}_H = \lambda_V$,
  \item $\tilde{\tau}_V = \nu_V p_0$ and $\nu_V \tilde{\tau}_H = \lambda_0$,
  \item $\nu_V \tilde{p}_V = \lambda_H$, $\tilde{p}_H = p_V$ and $\nu_V \tilde{p}_0 = \tau_H$,
  \item $\tilde{p}_V + \tilde{p}_H + \tilde{p}_0 = p_V + p_H + p_0$.
  \end{enumerate}
  Then, for any $\nu_H \in \RR_{\geq 0}$, for any $a,b \in \RR_{\geq 0}$,
  \begin{displaymath}
    s_{\pi/2} \left( \mathcal{L}_{[0,\infty)^2}\left(\para,N_{\PPP(\nu_H,\nu_V)}\right)|_{[0,a] \times [0,b]} \right) \eqd \mathcal{L}_{[0,\infty)^2}\left(\tilde{\para},N_{\PPP(\nu_V,\nu_H)}\right)|_{[0,b] \times [0,a]}. 
  \end{displaymath}
  
  Moreover, if there exists $\nu_H \in \RR_{\geq 0}$ such that the measure $N_{\PPP(\nu_H,\nu_V)}$ is stationary for the \Bullet model with parameter $\para$, then the measure $N_{\PPP(\nu_V,\nu_H)}$ is stationary for the \Bullet model with parameter $\tilde{\para}$; and the \Bullet model with parameter $\para$ under the stationary measure $N_{\PPP(\nu_H,\nu_V)}$ is $s_{\pi/2}$-quasi-reversible and its $s_{\pi/2}$-reverse is $(\tilde{\para},N_{\PPP(\nu_V,\nu_H)})$.
\end{theorem}

Combining Theorem~\ref{thm:pi2} with Proposition~\ref{prop:r} permits to obtain sufficient conditions for quasi-reversibility of the last three elements of $D_4$: $r \circ s_{\pi/2}$, $s_{\pi/2} \circ r = r \circ s_{3\pi/2}$ and $r \circ s_{\pi/2} \circ r = s_{3\pi/2}$.

Now, let us state its corollary that involves only one \Bullet model with parameter $\para$.
\begin{corollary} \label{cor:pi2}
  Consider a \Bullet model with parameter $\para$ such that $A,B_V,B_H \neq 0$ (as defined in Equations~\eqref{eq:A},~\eqref{eq:BV} and~\eqref{eq:BH}) and
  \begin{equation} \label{eq:pi2-rev}
    B_V \lambda_0 = A \tau_V \text{ and } A p_0 = B_H \tau_V.
  \end{equation}
  Then, by Corollary~\ref{cor:pi}, the measure $N_{\PPP(\nu_H,\nu_V)}$ is stationary for it with $(\nu_H,\nu_V)$ as defined in Equation~\eqref{eq:nu} and the \Bullet model with parameter $\para$, under this law, is $s_\pi$-quasi-reversible.\par
  Moreover, the couple $(\para,N_{\PPP(\nu_H,\nu_V)})$ is $s_{\pi/2}$-quasi-reversible and its $s_{\pi/2}$-reverse is $(\tilde{\para},N_{\PPP(\nu_V,\nu_H)})$ where the parameter $\tilde{\para}$ is
  \begin{itemize}
  \item $\tilde{\lambda}_0 = \frac{A}{B_V} \tau_V (= \lambda_0)$, $\tilde{\lambda}_V = \frac{A}{B_V} p_H$ and $\tilde{\lambda}_H = \lambda_V$,
  \item $\tilde{\tau}_V = \frac{B_H}{B_V} \tau_V (= \frac{A}{B_V} p_0)$ and $\tilde{\tau}_H = \tau_V (= \frac{B_V}{A} \lambda_0)$,
  \item $\tilde{p}_V = \frac{B_V}{A} \lambda_H$, $\tilde{p}_H = p_V$ and $\tilde{p}_0 = \frac{B_V}{A} \tau_H$.
  \end{itemize}
\end{corollary}

\begin{remark}
  Note that Equation~\eqref{eq:pi2-rev} implies Equation~\eqref{eq:pi-rev}.
\end{remark}

\subsection{Applications to some examples} \label{sec:appli}
\paragraph{Application to the Colliding \BULLET Model with Creations:}
Corollary~\ref{cor:pi} applies to the \Bullet model with parameter $\para_V$. Hence, $N_{\PPP(1,1)}$ is one of its stationary measures, and under this one, it is $s_\pi$-quasi-reversible and its $s_{\pi}$-reverse is $(\para_H,N_{\PPP(1,1)})$. Combined with Proposition~\ref{prop:r}, we can even deduce that the \Bullet model with parameter $\para_V$ under its stationary measure $N_{\PPP(1,1)}$ is $(s_{\pi} \circ r)$-reversible.\par
On the contrary, Corollary~\ref{cor:pi2} does not apply to the \Bullet model with parameter $\para_V$, but applies to the \Bullet model with parameter $\para_H$. Its application induces that
\begin{proposition} \label{prop:CBMCqr}
  The $s_{\pi/2}$-reverse of $(\para_H,N_{\PPP(1,1)})$ is $(\para_{\TABM},N_{\PPP(1,1)})$.
\end{proposition}
In particular, this proves Proposition~\ref{prop:TABM}.\par

\paragraph{Application to the \Bullet model with parameter $\para = (0,1,1,0,0,1/2,1/2,0)$:}
Corollary~\ref{cor:pi2} applies to the \Bullet model with parameter $\para = (0,1,1,0,0,1/2,1/2,0)$. Hence, $N_{\PPP(2,2)}$ is one of its stationary measures, and, under it, it is $s_{\pi/2}$-reversible. As, it is also $r$-reversible, we obtain that
\begin{proposition} \label{prop:loop}
  One of the stationary measures of the \Bullet model with parameter $\para = (0,1,1,0,0,1/2,1/2,0)$ is the measure $N_{\PPP(2,2)}$. Moreover, for any symmetry $g \in D_4$, the couple $(\para,N_{\PPP(2,2)})$ is $g$-reversible.
\end{proposition}

In~\cite{BCGKP15}, they proved that the stationary measure is two invariant PPP with intensities $1$. The difference of intensities is explained by two factors $\sqrt{2}$: one comes from the rotation of the model, the other one comes from the fact that they studied the stationary measure on the anti-diagonal whereas, in this article, the stationary measure is considered on $x$-and $y$-axes.

\paragraph{Application to the \Bullet model with parameter $\para_{\text{loop}}$:}
Neither Corollary~\ref{cor:pi}, neither Corollary~\ref{cor:pi2} applies to the \Bullet model with parameter $\para_{\text{loop}}$ because $A = B_V= B_H =0$. Nevertheless, taking $\tilde{\para} = \para_{\text{loop}}$ and $\nu_V = \nu_H =1$, both Theorem~\ref{thm:pi} and Theorem~\ref{thm:pi2} hold. Hence, one of the stationary measures of the \Bullet model with parameter $\para_{\text{loop}}$ is the measure $N_{\PPP(1,1)}$. Moreover, under it, the \Bullet model is $s_\pi$-reversible and $s_{\pi/2}$-reversible. In fact, because it is also $r$-reversible, we get that
\begin{proposition} \label{prop:loop}
  One of the stationary measures of the \Bullet model with parameter $\para_{\text{loop}}$ is the measure $N_{\PPP(1,1)}$. Moreover, for any symmetry $g \in D_4$, the couple $(\para_{\text{loop}},N_{\PPP(1,1)})$ is $g$-reversible.
\end{proposition}

\section{Heuristic of the quasi-reversibility} \label{sec:heu}
The heuristics are quite simple to understand. We use a similar heuristic than in~\cite{BCES23}, but the idea is slightly improved to be more precise. First, the heuristic of the $s_{\pi}$-quasi-reversibility is given and, then, the one of the $s_{\pi/2}$-quasi-reversibility that is slightly more complicated.

\subsection{Heuristic for $s_\pi$} \label{sec:heupi}
Take two \Bullet models, one with parameter $\para$ and the other one with parameter $\tilde{\para}$, that have the probability measure $N_{\PPP(\nu_H,\nu_V)}$ as a common stationary measure. Suppose that the two \Bullet models under $N_{\PPP(\nu_H,\nu_V)}$ are $s_\pi$-reverse. Then, if $A$ is an event inside the rectangle $[-a,a] \times [-b,b]$ that occurs with probability $p$ under the probability $\rho(\para,N_{\PPP(\nu_H,\nu_V)})$, then the reverse event $s_\pi(A)$ must occur with probability $p$ under the probability $\rho(\tilde{\para},N_{\PPP(\nu_H,\nu_V)})$.\par
\medskip
In particular, if we consider the elementary event $A$, represented on line~4 of Table~\ref{tab:heupi}, that corresponds to a unique horizontal particle that turns and becomes a vertical particle at position $(x,y)$ in the rectangle $[-a,a] \times [-b,b]$. Its ``probability'' (in reality, its Radon--Nykodim density under a well-chosen measure on the set of segments inside the rectangle $[-a,a] \times [-b,b]$, see Section~\ref{sec:dens} for the technical details) under $\rho(\para,N_{\PPP(\nu_H,\nu_V)})$ is
\begin{equation} \label{eq:hvp}
  \underbrace{\nu_H \ind{-b < y < b} e^{-2 \nu_H b}}_{E_1} \underbrace{e^{-2 \nu_V a}}_{E_2} \underbrace{e^{-(\lambda_H+\tau_H) (a+x)}}_{E_3}  \underbrace{\tau_H \ind{-a<x<a}}_{E_4} \underbrace{e^{-(\lambda_V+\tau_V)(b-y)}}_{E_5} \underbrace{e^{-4 \lambda_0 a b}}_{E_6}.
\end{equation}
Let us explain Equation~\eqref{eq:hvp}.
\begin{itemize}
\item The term $E_1$ corresponds to the probability of a unique horizontal particle that enters at the left of the rectangle at position ordinate $y$.
\item The term $E_2$ corresponds to the probability of none vertical particle enters at the bottom of the rectangle.
\item The term $E_3$ corresponds to the probability of the horizontal particle to do not create a horizontal particle on $[-a,x] \times \{y\}$.
\item The term $E_4$ corresponds to the probability of a horizontal particle to turn inside the rectangle at abscissa $x$.
\item The term $E_5$ corresponds to the probability of the vertical particle to do not create a horizontal particle on $\{x\} \times [y,b]$.
\item The term $E_6$ corresponds to the probability that there is no ex-nihilo creation in the rectangle.
\item As those events are independent, we multiply all of them.
\end{itemize}

Similarly, the reverse event $s_\pi(A)$ that corresponds to a vertical particle that becomes a horizontal particle at position $(-x,-y)$ has ``probability''
\begin{equation}
  \underbrace{\nu_V \ind{-a < -x < a} e^{-2 \nu_V a}} \underbrace{e^{-2 \nu_H b}}   \underbrace{e^{-(\tilde{\lambda}_V+\tilde{\tau}_V) (b-y)}} \underbrace{\tilde{\tau}_V \ind{-b < -y < b}} \underbrace{e^{-(\tilde{\lambda}_H+\tilde{\tau}_H)(a+x)}} \underbrace{e^{-4 \tilde{\lambda}_0 a b}}.
\end{equation}

If we suppose that $\lambda_0 = \tilde{\lambda}_0$, $\lambda_V+\tau_V = \tilde{\lambda}_V + \tilde{\tau}_V$, $\lambda_H + \tau_H = \tilde{\lambda}_H + \tilde{\tau}_H$ and $\nu_H \tau_H = \nu_V \tilde{\tau}_V$, the equality is satisfied.

\begin{table}[p]
  \begin{center}
  \begin{tabular}{|m{3cm}|m{5cm}||m{5cm}|m{3cm}|}
    \hline
    Configuration & probability & $s_\pi$ probability & $s_\pi$-configuration \\ \hline
    \includegraphics{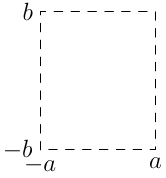} & $e^{-2 \nu_H b} e^{-2 \nu_V a} e^{-4 \red{\lambda_0} ab}$ & $e^{-2 \nu_H b} e^{-2 \nu_V a} e^{-4 \red{\tilde{\lambda}_0} ab}$ & \includegraphics{DESSIN/Heuristique/Vide.pdf} \\ \hline
    \includegraphics{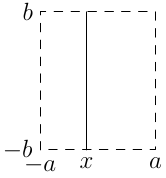} & $\nu_V \ind{-a< x < a}\, e^{-2 \red{(\lambda_V+\tau_V)} b}$ $e^{-2 \nu_H b} e^{-2 \nu_V a} e^{-4 \lambda_0 ab}$ & $\nu_V \ind{-a<-x<a}\, e^{-2 \red{(\tilde{\lambda}_V+\tilde{\tau}_V)} b}$ $e^{-2 \nu_H b} e^{-2 \nu_V a} e^{-4 \tilde{\lambda}_0 ab}$ & \includegraphics{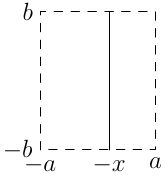} \\ \hline
    \includegraphics{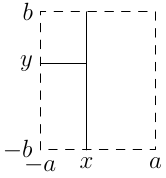} & $\red{p_V}\,  \red{\nu_H} \ind{-b< y < b}\, e^{-(\lambda_H+\tau_H) (a+x)}$ $\nu_V \ind{-a< x < a}\, e^{-2 (\lambda_V+\tau_V) b}$  $e^{-2 \nu_H b} e^{-2 \nu_V a} e^{-4 \lambda_0 ab}$ & $\nu_V \ind{-a<-x<a} \, e^{-2 (\tilde{\lambda}_V+\tilde{\tau}_V) b}$ $\red{\tilde{\lambda}_V} \ind{-b<-y<b}\, e^{-(\tilde{\lambda}_H+\tilde{\tau}_H) (a+x)}$ $e^{-2 \nu_H b} e^{-2 \nu_V a} e^{-4 \tilde{\lambda}_0 ab}$ & \includegraphics{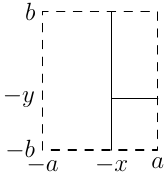} \\ \hline
    \includegraphics{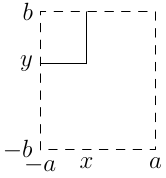} & $\red{\nu_H} \ind{-b< y < b} \, e^{-(\lambda_H+\tau_H) (a+x)}$ $\red{\tau_H} \ind{-a< x < a} \, e^{-(\lambda_V+\tau_V) (b-y)}$ $e^{-2 \nu_H b} e^{-2 \nu_V a} e^{-4 \lambda_0 ab}$ & $\red{\nu_V} \ind{-a< -x < a} \, e^{-(\tilde{\lambda}_V+\tilde{\tau}_V) (b-y)}$ $\red{\tilde{\tau}_V} \ind{-b< -y < b} \, e^{-(\tilde{\lambda}_H+\tilde{\tau}_H) (a+x)}$ $e^{-2 \nu_H b} e^{-2 \nu_V a} e^{-4 \tilde{\lambda}_0 ab}$ & \includegraphics{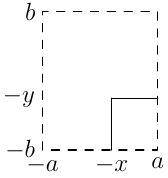}\\ \hline
    \includegraphics{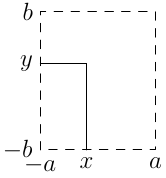} & $\red{p_0}\, \red{\nu_H} \ind{-b< y < b} \, e^{-(\lambda_H+\tau_H) (a+x)}$ $\red{\nu_V} \ind{-a< x < a} \, e^{-(\lambda_V+\tau_V) (b+y)}$ $e^{-2 \nu_H b} e^{-2 \nu_V a} e^{-4 \lambda_0 ab}$ & $\red{\tilde{\lambda}_0} \ind{-a< -x < a} \ind{-b<-y<b}$ $e^{-(\tilde{\lambda}_H+\tilde{\tau}_H) (a+x)} e^{-(\tilde{\lambda}_V+\tilde{\tau}_V) (b+y)}$ $e^{-2 \nu_H b} e^{-2 \nu_V a} e^{-4 \tilde{\lambda}_0 ab}$ & \includegraphics{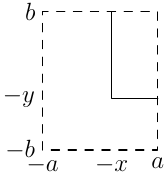}\\ \hline
    \includegraphics{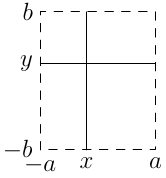} & $\red{(1-p_V-p_H-p_0)}$ $\nu_H \ind{-b< y < b} \, e^{-2(\lambda_H+\tau_H) a}$ $\nu_V \ind{-a< x < a}\, e^{-2(\lambda_V+\tau_V)b}$ $e^{-2 \nu_H b} e^{-2 \nu_V a} e^{-4 \lambda_0 ab}$ & $\red{(1-\tilde{p}_V-\tilde{p}_H-\tilde{p}_0)}$ $\nu_H \ind{-b< -y < b}\, e^{-2(\tilde{\lambda}_H+\tilde{\tau}_H) a}$ $\nu_V \ind{-a< -x < a}\, e^{-2(\tilde{\lambda}_V+\tilde{\tau}_V)b}$ $e^{-2 \nu_H b} e^{-2 \nu_V a} e^{-4 \tilde{\lambda}_0 ab}$ & \includegraphics{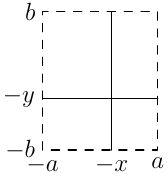}\\ \hline
  \end{tabular}
\end{center}
\caption{The heuristics to find the conditions of Theorem~\ref{thm:pi}.}%
\label{tab:heupi}
\end{table}
\medskip

Considering all the elementary configurations given in Table~\ref{tab:heupi} and similar events (reversing the $x$-and $y$-axes), we obtain the five conditions expressed in Theorem~\ref{thm:pi}. Proof of Theorem~\ref{thm:pi} consists of checking that those conditions are sufficient for any configuration and not just for the elementary ones. This is done in Section~\ref{sec:proof-pi}.

\subsection{Heuristic for $s_{\pi/2}$} \label{sec:heupi2}
The heuristic for $s_{\pi/2}$ is quite more complex because the stationary measure for one of them is not explicit. To get around this difficulty, we consider quotient of events that have the same initial condition. Let us illustrate the idea with the case that a unique vertical particle enters in $(x,-b)$ at the bottom of the rectangle $[-a,a] \times [-b,b]$, see Table~\ref{tab:heupi2-elem} for the elementary cases without ex-nihilo creation and Table~\ref{tab:heupi2-complex} for elementary cases with an ex-nihilo creation. A study, similar to the one done for the elementary cases in $s_\pi$-quasi-reversibility, suggests that some good sufficient conditions are the five expressed in Theorem~\ref{sec:pi2}. Moreover, it suggests also that $\tilde{N} = N_{\PPP(\nu_V,\nu_H)}$.

\begin{table}
  \begin{tabular}{|m{3cm}|m{5cm}||m{5cm}|m{3cm}|}
    \hline 
    \includegraphics{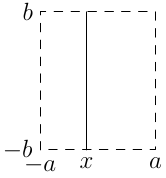} & $e^{-2\red{(\tilde{\lambda}_V+\tilde{\tau}_V)}b}$ $\grey{\tilde{N}(x) e^{-4 \tilde{\lambda}_0 ab}}$ & $\grey{\nu_H \ind{-a<-x<a}}\, e^{-2 \red{(\lambda_H+\tau_H)} b}$ $\grey{e^{-2 \nu_H a} e^{-2 \nu_V b} e^{-4 \lambda_0 ab}}$ & \includegraphics{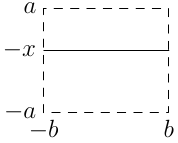} \\ \hline 
    \includegraphics{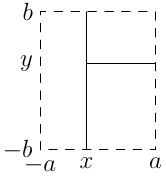} & $\red{\tilde{\lambda}_V} \ind{-b<y<b}\, e^{-(\tilde{\lambda}_H+\tilde{\tau}_H)(a-x)}$ $e^{-2(\tilde{\lambda}_V+\tilde{\tau}_V)b}$ $\grey{\tilde{N}(x) e^{-4 \tilde{\lambda}_0 ab}}$ & $\red{p_H} \, \grey{\nu_H \ind{-a<-x<a}}\, e^{-2 (\lambda_H+\tau_H) b}$ $\red{\nu_V} \ind{-b<y<b}\, e^{- (\lambda_V+\tau_V) (a-x)}$ $\grey{e^{-2 \nu_H a} e^{-2 \nu_V b} e^{-4 \lambda_0 ab}}$ & \includegraphics{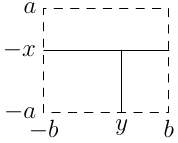} \\ \hline 
    \includegraphics{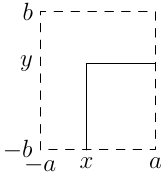} & $\red{\tilde{\tau}_V} \ind{-b<y<b}\, e^{-(\tilde{\lambda}_H+\tilde{\tau}_H)(a-x)}$ $e^{-(\tilde{\lambda}_V+\tilde{\tau}_V)(y+b)}$ $\grey{\tilde{N}(x) e^{-4 \tilde{\lambda}_0 ab}}$ & $\red{p_0} \, \grey{\nu_H \ind{-a<-x<a}}\, e^{-(\lambda_H+\tau_H) (y+b)}$ $\red{\nu_V} \ind{-b<y<b}\, e^{- (\lambda_V+\tau_V) (a-x)}$ $\grey{e^{-2 \nu_H a} e^{-2 \nu_V b} e^{-4 \lambda_0 ab}}$ & \includegraphics{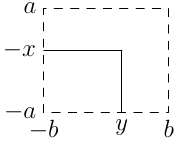} \\ \hline 
   \includegraphics{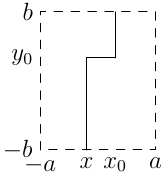} & $\blue{\tilde{\tau}_V} \ind{-b<y_0<b} \, \red{\tilde{\tau}_H} \ind{x<x_0<a}$ $e^{-(\tilde{\lambda}_H+\tilde{\tau}_H)(x_0-x)} e^{-2(\tilde{\lambda}_V+\tilde{\tau}_V)b}$ $\grey{\tilde{N}(x) e^{-4 \tilde{\lambda}_0 ab}}$ & $\blue{p_0} \red{\lambda_0} \ind{-b<y_0<b} \ind{-a<-x_0<-x}$ $\grey{\nu_H \ind{-a<-x<a}}\, e^{-2(\lambda_H+\tau_H)b}$ $\ind{-b<y<b}\, e^{- (\lambda_V+\tau_V) (x_0-x)}$ $\grey{e^{-2 \nu_H a} e^{-2 \nu_V b} e^{-4 \lambda_0 ab}} \frac{\blue{\nu_V}}{\red{\nu_V}}$ & \includegraphics{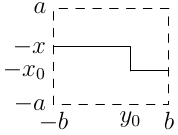}\\ \hline
  \end{tabular}
  \caption{Heuristic for the $s_{\pi/2}$-quasi-reversibility: four elementary cases.}%
  \label{tab:heupi2-elem}
\end{table}

\begin{table}
  \begin{tabular}{|m{3cm}|m{5cm}||m{5cm}|m{3cm}|}
    \hline
    \includegraphics{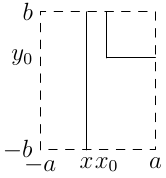} & $e^{-2(\tilde{\lambda}_V+\tilde{\tau}_V)b}$ $\red{\tilde{\lambda}_0} \ind{x<x_0<a} \ind{-b<y_0<b}$ $e^{-(\tilde{\lambda}_V+\tilde{\tau}_V)(b-y_0)} e^{-(\tilde{\lambda}_H+\tilde{\tau}_H)(a-x_0)}$ $\grey{\tilde{N}(x) e^{-4 \tilde{\lambda}_0 ab}}$ & $\grey{\nu_H \ind{-a<-x<a}}\, e^{-2 (\lambda_H+\tau_H) b}$ $\red{\nu_V} \ind{-b<y_0<b}\, e^{- (\lambda_V+\tau_V) (a-x_0)}$ $\red{\tau_V} \ind{-a<-x_0<-x}\, e^{-(\lambda_H+\tau_H)(b-y_0)}$ $\grey{e^{-2 \nu_H a} e^{-2 \nu_V b} e^{-4 \lambda_0 ab}}$ & \includegraphics{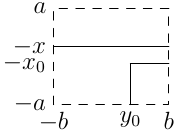} \\ \hline 
    \includegraphics{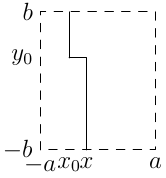} & $\red{\tilde{p}_0} \blue{\tilde{\lambda}_0} \ind{-a<x_0<x} \ind{-b<y_0<b}$ $e^{-(\tilde{\lambda}_H+\tilde{\tau}_H)(x-x_0)} e^{-2(\tilde{\lambda}_V+\tilde{\tau}_V)b}$ $\grey{\tilde{N}(x) e^{-4 \tilde{\lambda}_0 ab}}$ & $\grey{\nu_H \ind{-a<-x<a}}\, \red{\tau_H} \ind{-b<y_0<b}$ $\blue{\tau_V}  \ind{-x<-x_0<a}$ $e^{-2 (\lambda_H+\tau_H) b} e^{- (\lambda_V+\tau_V) (a-x_0)}$ $\grey{e^{-2 \nu_H a} e^{-2 \nu_V b} e^{-4 \lambda_0 ab}} \frac{\blue{\nu_V}}{\red{\nu_V}}$ & \includegraphics{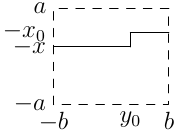} \\ \hline 
    \includegraphics{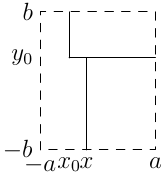} & $\red{\tilde{p}_H} \blue{\tilde{\lambda}_0} \ind{-a<x_0<x} \ind{-b<y_0<b}$ $e^{-2(\tilde{\lambda}_V+\tilde{\tau}_V)b} e^{-(\tilde{\lambda}_H+\tilde{\tau}_H)(a-x_0)}$ $\grey{\tilde{N}(x) e^{-4 \tilde{\lambda}_0 ab}}$ & $\grey{\nu_H \ind{-a<-x<a}}\, \blue{\nu_V} \ind{-b<y_0<b}$ $\red{p_V} \blue{\tau_V}  \ind{-x<-x_0<a}$ $e^{-2 (\lambda_H+\tau_H) b} e^{- (\lambda_V+\tau_V) (x-x_0)}$ $\grey{e^{-2 \nu_H a} e^{-2 \nu_V b} e^{-4 \lambda_0 ab}}$ & \includegraphics{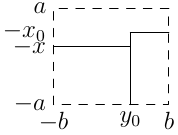} \\ \hline 
    \includegraphics{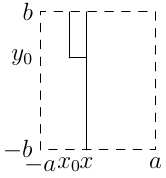} & $\red{\tilde{p}_V} \blue{\tilde{\lambda}_0} \ind{-a<x_0<x} \ind{-b<y_0<b}$ $e^{-(\tilde{\lambda}_H+\tilde{\tau}_H)(x-x_0)} e^{-(\tilde{\lambda}_V+\tilde{\tau}_V)(b-y_0)}$ $e^{-2(\tilde{\lambda}_V+\tilde{\tau}_V)b}$ $\grey{\tilde{N}(x) e^{-4 \tilde{\lambda}_0 ab}}$ & $\grey{\nu_H \ind{-a<-x<a}}\, e^{-2 (\lambda_H+\tau_H) b}$ $\red{\lambda_H} \ind{-b<y_0<b} \, e^{- (\lambda_V+\tau_V) (x-x_0)}$ $\blue{\tau_V} \ind{-x<-x_0<a} \, e^{- (\lambda_H+\tau_H) (b-y_0)}$ $\grey{e^{-2 \nu_H a} e^{-2 \nu_V b} e^{-4 \lambda_0 ab}} \frac{\blue{\nu_V}}{\red{\nu_V}}$ & \includegraphics{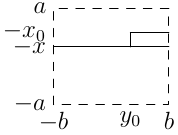} \\ \hline 
    \includegraphics{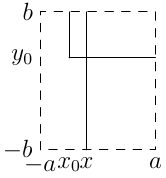} & $\red{(1-\tilde{p}_0-\tilde{p}_H-\tilde{p}_V)}$ $\blue{\tilde{\lambda}_0} \ind{-a<x_0<x} \ind{-b<y_0<b}$ $e^{-(\tilde{\lambda}_H+\tilde{\tau}_H)(a-x_0)} e^{-(\tilde{\lambda}_V+\tilde{\tau}_V)(b-y_0)}$ $e^{-2(\tilde{\lambda}_V+\tilde{\tau}_V)b} \grey{\tilde{N}(x) e^{-4 \tilde{\lambda}_0 ab}}$ & $\grey{\nu_H \ind{-a<-x<a}}\, \blue{\nu_V} \ind{-b<y_0<b}$ $\red{(1-p_0-p_H-p_V)}$ $\blue{\tau_V}  \ind{-x<-x_0<a}\, e^{-(\lambda_H+\tau_H) (b-y_0)}$ $e^{-2 (\lambda_H+\tau_H) b} e^{- (\lambda_V+\tau_V) (a-x_0)}$ $\grey{e^{-2 \nu_H a} e^{-2 \nu_V b} e^{-4 \lambda_0 ab}}$ & \includegraphics{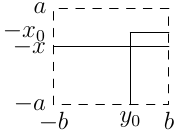} \\ \hline 
  \end{tabular}
  \caption{Heuristic for the $s_{\pi/2}$-quasi-reversibility: five complex cases.}%
  \label{tab:heupi2-complex}
\end{table}

\section{Proofs of Theorems~\ref{thm:pi} and~\ref{thm:pi2}} \label{sec:proof}
The proofs of Theorems~\ref{thm:pi} and~\ref{thm:pi2} follow the idea developed in~\cite{BCES23}. As explained in the heuristic, there is a first step that consists in the construction of the densities of \Bullet models on any rectangle $[-a,a] \times [-b,b]$, see Section~\ref{sec:dens}. When this density is constructed, it suffices to check that the density of the \Bullet model with parameter $\para$ under the initial condition distributed according to $N_{\PPP(\nu_H,\nu_V)}$ is the density of the \Bullet model with parameter $\tilde{\para}$ under the initial condition distributed according to $N_{\PPP(\nu_H,\nu_V)}$ or to $N_{\PPP(\nu_V,\nu_H)}$ after a rotation of angle $\pi$ or $\pi/2$ of the plane. This is done respectively in Sections~\ref{sec:proof-pi} and~\ref{sec:proof-pi2}.

\subsection{Construction of the density of a \Bullet model} \label{sec:dens}
In this section, we define formally the probability densities presented in Section~\ref{sec:heu}. This induces many formal definitions.

\paragraph{The set of configurations:}
We focus, first, on the set of configurations of a \Bullet model inside any rectangle $[-a,a]\times[-b,b]$.\par

Let $a,b$ be two positive real numbers. We define the two sets $\set_V$ and $\set_H$ of vertical and horizontal segments inside the rectangle $[-a,a] \times [-b,b]$: 
\begin{align}
  & \set_V = \{[(x,y),(x,y+t)] : -a < x < a , -b \leq y<y+t \leq b\},\\
  & \set_H = \{[(x,y),(x+t,y)] : -a \leq x < x+t \leq a , -b<y<b\}.
\end{align}
Now, we define the set of configurations $\mathcal{U}$ of any realisation of a \Bullet model inside the rectangle $[-a,a]\times[-b,b]$:
\begin{align} \label{eq:config}
  \mathcal{U} & = \{ U \subset \mathcal{P}(\set_V \cup \set_H) : |U|< \infty, \nonumber\\
              & \qquad \text{if } S = [(x,y_0),(x,y_1)] \in U \cap \set_V\text{, then}\ (y_0=-b \text{ or } \exists S' \in U \cap \set_H,\ (x,y_0) \in S') \text{ and}\nonumber\\
              & \qquad \qquad (y_1 = b \text{ or } \exists S'' \in U \cap \set_H,\ (x,y_1) \in S''),\nonumber\\
              & \qquad \text{if } S = [(x_0,y),(x_1,y)] \in U \cap \set_H\text{, then}\ (x_0=-a \text{ or } \exists S' \in U \cap \set_V,\ (x_0,y) \in S') \text{ and}\nonumber\\
              & \qquad \qquad (x_1 = a \text{ or } \exists S'' \in U \cap \set_V,\ (x_1,y) \in S''),\nonumber\\
              & \qquad \text{if } S = [(x,y_0),(x,y_1)],\ S'= [(x',y'_0),(x',y'_1)] \in U \cap \set_V \text{ with } S \neq S' \text{, then}\ x \neq x',\nonumber\\
              & \qquad \text{if } S = [(x_0,y),(x_1,y)],\ S'= [(x'_0,y'),(x'_1,y')] \in U \cap \set_H \text{ with } S \neq S' \text{, then}\ y \neq y'\}.
\end{align}
In words, this formalism tells that:
\begin{itemize}
\item any realisation is a finite set of vertical and horizontal segments,
\item a vertical segment in the set begins at the bottom edge of the rectangle $[-a,a] \times [-b,b]$ or at a horizontal segment in the set and 
\item ends at the top edge of the rectangle or at a horizontal segment in the set,
\item similarly for the beginnings and the ends of horizontal segments,
\item the two last conditions are technical: they express the fact that two vertical, resp.\ horizontal, different segments do not share the same abscissa, resp.\ ordinate. It permits to assure the uniqueness of vertical, resp.\ horizontal, segments. Indeed, when a vertical, resp.\ horizontal, particle creates, or kills, or crosses a horizontal, resp.\ vertical, one, without these conditions, we could split its representation in many segments. The disadvantage of these conditions is that it forbids two different vertical, resp.\ horizontal, particles to share the same abscissa, resp.\ ordinate. Hopefully, in the random \Bullet model considered here, this never happens almost surely.
\end{itemize}

\paragraph{Equivalence relation, the skeleton, on $\mathcal{U}$:}
Let $U$ be an element of $\mathcal{U}$. We denote by $n=|U \cap \set_V|$, the number of vertical segments in $U$ and by $m = |U \cap \set_H|$ the number of horizontal segments in $U$. Due to the last two conditions of Equation~\eqref{eq:config}, we can order them by their abscissas or by their ordinates. Hence, we denote $x_1, \dots, x_n$ the abscissas of vertical ones such that $-a<x_1<x_2<\dots<x_n<a$ and by $y_1,\dots,y_m$ the ordinates of the horizontal ones with $-b<y_1<\dots<y_m<b$.\par

Take two elements $U,U' \in \mathcal{U}$. We say that they share the same skeleton, denoted $U \sk U'$ if there exists two one-to-one and increasing functions, one $\phi_x$ from $[-a,a]$ to $[-a,a]$ and the other one $\phi_y$ from $[-b,b]$ to $[-b,b]$, such that $\{[(\phi_x(x_0),\phi_y(y_0)),(\phi_x(x_1),\phi_y(y_1))] : S =[(x_0,y_0),(x_1,y_1)] \in U'\} = U$. It is an equivalent relation on $\mathcal{U}$.\par

Let denote by $\mathcal{K} = \mathcal{U}/\sk$ the set of skeletons. A canonical representation for a skeleton is to take the element $U$ in that class such that $x_i = -a + i \frac{2a}{n+1}$ and $y_i = -b + i \frac{2b}{m+1}$, see Figure~\ref{fig:skelette}.\par

\begin{figure}
  \begin{center}
    \begin{tabular}{cc}
      \includegraphics[width = 0.45 \textwidth]{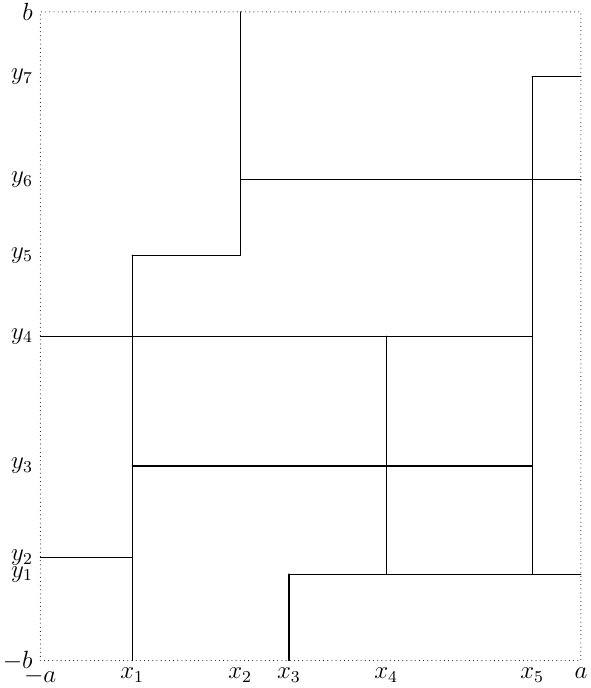} & \includegraphics[width = 0.45 \textwidth]{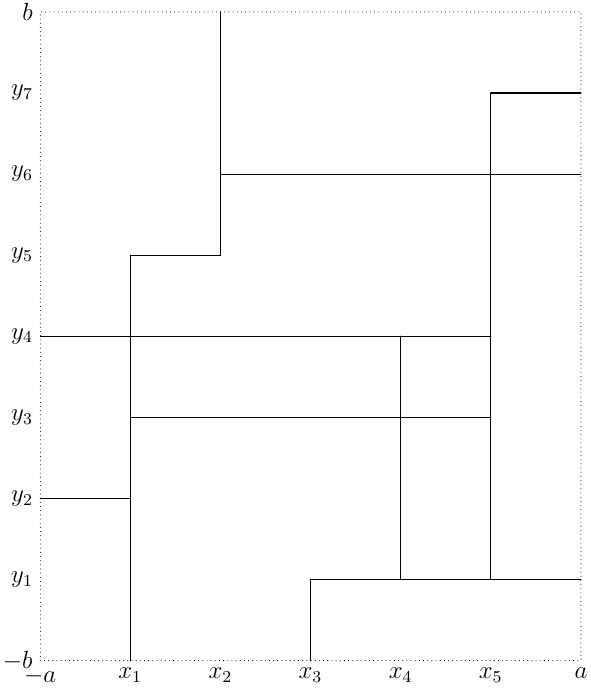}
    \end{tabular}
  \end{center}
  \caption{On the left: an element of $\mathcal{U}$. On the right: another element of $\mathcal{U}$ with the same skeleton, but with different coordinates in their equivalent class $K(U)$.}%
  \label{fig:skelette}
\end{figure}

Now, any element $U$ of $\mathcal{U}$ can be uniquely represented by its skeleton $K(U)$, the abscissas $x(U) = (x_1,\dots,x_n)$ of its vertical segments and the ordinates $y(U) = (y_1,\dots,y_m)$ of its horizontal ones. We call coordinates of $U$ in $K(U)$ the uplet $(x_1,\dots,x_n;y_1,\dots,y_m)$.\par

\paragraph{Shared statistics by elements with the same skeleton:}
Take $U$ and $U'$ two elements of $\mathcal{U}$ with the same skeleton $K(U) = K(U') = K$. They share many common statistics. The first two obvious are
\begin{itemize}
\item the number $n(U)$ of vertical segments and
\item the number $m(U)$ of horizontal segments.
\end{itemize}
But, as the relative positions of vertical and horizontal segments are the same, we can observe $13$ types of points whose cardinals are the same, see also~\cite[Sections~5.1~and~5.2]{BCES23}:
\begin{itemize}
\item the number of vertical entries
  \begin{displaymath}
    \VE(U) =  \card{\{x \in (-a,a) : \exists y_2 \in (-b,b] : [(x,-b),(x,y_2)] \in U\}},
  \end{displaymath}

\item the number of vertical exits
  \begin{displaymath}
    \VS(U) =  \card{\{x \in (-a,a) : \exists y_1 \in [-b,b) : [(x,y_1),(x,b)] \in U\}},
  \end{displaymath}

\item the number of vertical splits
  \begin{align*}
    \HB(U) = \text{card} (\{ (x,y) \in (-a,a) \times (-b,b) :\ & \exists x_2 \in (x,a],\, [(x,y),(x_2,y)] \in U \text{ and } \\
                                                               &\exists y_1 \in [-b,y),\, \exists y_2 \in (y,b],\, [(x,y_1),(x,y_2)] \in U \} ),
  \end{align*}
  
\item the number of vertical turns
  \begin{align*}
    \HT(U) = \text{card} (\{(x,y) \in (-a,a) \times (-b,b) :\ & \exists x_2 \in (x,a],\, [(x,y),(x_2,y)] \in U \text{ and } \\
                                                              & \exists y_1 \in [-b,y),\, [(x,y_1),(x,y)] \in U \} ),
  \end{align*}
  
\item the number of vertical coalescences 
  \begin{align*}
    \HA(U) = \text{card} ( \{(x,y) \in (-a,a) \times (-b,b) :\ & \exists x_1 \in [-a,x),\, [(x_1,y),(x,y)] \in U \text{ and } \\
                                                               & \exists y_1 \in [-b,y),\, \exists y_2 \in (y,b],\, [(x,y_1),(x,y_2)] \in U \} ),
  \end{align*}
  
\item all the horizontal counterparts of the previous ones: the numbers of horizontal entries $\HE(U)$, exits $\HS(U)$, splits $\VB(U)$, turns $\VT(U)$ and coalescences $\VA(U)$,
  
\item the number of crossings
  \begin{align*}
    \CC(U) = \text{card} ( \{(x,y) \in (-a,a) \times (-b,b) :\ & \exists x_1 \in [-a,x),\, \exists x_2 \in (x,a],\, [(x_1,y),(x_2,y)] \in U \text{ and } \\
                                                               & \exists y_1 \in [-b,y),\, \exists y_2 \in (y,b],\, [(x,y_1),(x,y_2)] \in U \} ),
  \end{align*}
  
\item the number of spontaneous splits (ex-nihilo creations)
  \begin{align*}
    \OB(U) = \text{card} (\{(x,y) \in (-a,a) \times (-b,b) :\ & \exists x_2 \in (x,a],\, [(x,y),(x_2,y)] \in U \text{ and } \\
                                                              & \exists y_2 \in (y,b],\, [(x,y),(x,y_2)] \in U \} ),
  \end{align*}

\item the number of double coalescences (total annihilations)
  \begin{align*}
    \OA(U) = \text{card} (\{(x,y) \in (-a,a) \times (-b,b) :\ & \exists x_1 \in [-a,x),\, [(x_1,y),(x,y)] \in U \text{ and } \\
                                                              & \exists y_1 \in [-b,y),\, [(x,y_1),(x,y)] \in U \} ),
  \end{align*}
\end{itemize}

As any of those elements depend on $U$ only through its skeleton $K(U)$, in the sequel, we write indifferently $\VE(U)$ or $\VE(K(U))$ depending on what we would like to insist.
  
\paragraph{A distance on $\mathcal{U}$ such that $\mathcal{U}$ is a Polish space:}
We define the distance $d$ on $\mathcal{U}$ by: for any $U,U' \in \mathcal{U}$,
\begin{equation}
  d(U,U') = \begin{cases}
    3 & \text{if } K(U) \neq K(U'),\\
    \displaystyle \ind{n \neq 0} \ \frac{1}{n} \sum_{i=1}^n \frac{ \big| x_i - x'_i \big|}{2a} + \ind{m \neq 0} \ \frac{1}{m} \sum_{j=1}^m \frac{ \big| y_j - y'_j \big|}{2b} & \text{else},
  \end{cases}
\end{equation}
where $(x_1,\dots,x_n;y_1,\dots,y_m)$ are the coordinates of $U$ in $K(U)$ and $(x'_1,\dots,x'_n;y'_1,\dots,y'_m)$ the ones of $U'$ in $K(U)=K(U')$.

\begin{lemma}
  The space of configurations $\mathcal{U}$ equipped with the distance $d$ is a Polish space.
\end{lemma}  

\begin{proof}
  The construction of the distance $d$ is made for that, if the distance is equal to $3$, then $U$ and $U'$ have different skeleton. On the contrary, if $U$ and $U'$ share the same skeleton, then $d(U,U') \leq 2$.\par
  So to check the completeness of $(\mathcal{U},d)$, we just have to check that, for any $K \in \mathcal{K}$. The space $K$ equipped with the distance
  \begin{equation}
    \ind{n(K) \neq 0} \ \frac{1}{n(K)} \sum_{i=1}^{n(K)} \frac{ \big| x_i - x'_i \big|}{2a} + \ind{m(K) \neq 0} \ \frac{1}{m(K)} \sum_{j=1}^{m(K)} \frac{ \big| y_j - y'_j \big|}{2b}
  \end{equation}
  is complete. This is the case because it is a modified version of the $L_1$-distance on a finite dimensional space.
\end{proof}

\paragraph{A $\sigma$-finite measure on $\mathcal{U}$:}
Let $K$ be any skeleton in $\mathcal{K}$. Take $U \in K$ of coordinate $(x_1,\dots,x_n;y_1,\dots,y_m)$. We define the ``uniform'' finite measure $\lambda_K$ on $K$ by
\begin{equation}
  \di \lambda_K(U) = \di \lambda_{K}(x_1,\dots,x_n;y_1,\dots,y_m) = \ind{-a< x_1<\dots<x_n<a}\, \ind{-b<y_1<\dots<y_m<b}\, \di x_1 \dots \di x_n\, \di y_1 \dots \di y_m.
\end{equation}

The set $\mathcal{K}$ of skeletons is enumerable and can be equipped with a $\sigma$-finite counting measure: for any $A \in \mathcal{P}(\mathcal{K})$, $\delta(A) = \sum_{K \in \mathcal{K}} \delta_K(A)$, where $\delta_K(A) = \ind{K \in A}$.\par
Finally, we define the measure $r$ on the Polish space $\mathcal{U}$ equipped with the distance $d$ by, for any Borel set $B$ of $\mathcal{U}$,
\begin{equation}
  r(B) = \sum_{K \in \mathcal{K}} \lambda_{K}(B \cap K).
\end{equation}

\begin{lemma}
  The measure $r$ is $\sigma$-finite.
\end{lemma}

\begin{proof}
  The space of configuration $\mathcal{U}$ is $\bigcup_{K \in \mathcal{K}} K$ and $\mathcal{K}$ is enumerable, so it is an enumerable union of spaces. Moreover, any $K$ is $r$-finite, its $r$-measure is
  \begin{align*}
    r(K) = \lambda_K(K) & = \left(\int_{(-a,a)^{n(K)}} \ind{x_1<\dots<x_{n(K)}} \di x_1 \dots \di x_{n(K)} \right) \left(\int_{(-b,b)^{m(K)}} \ind{y_1<\dots<y_{m(K)}} \di y_1 \dots \di y_{m(K)} \right)\\
    & = \frac{(2a)^{n(K)}}{n(K)!}  \frac{(2b)^{m(K)}}{m(K)!} < \infty. \qedhere
  \end{align*}
\end{proof}

\paragraph{\RN{} derivative:}
Let denote by $\tau(\para,N)$ the law of $\mathcal{L}_{[-a,+\infty) \times [-b,+\infty)]}(\para,N)|_{[-a,a] \times [-b,b]}$, see Section~\ref{sec:st} for a recall of this notation.\par
If there exists $(\nu_H,\nu_V) \in [0,\infty)^2$ such that $N = N_{\PPP(\nu_H,\nu_V)}$, then $\tau(\para,N_{\PPP(\nu_H,\nu_V)})$ is absolutely continuous according to the $\sigma$-finite measure~$r$, so we can define its \RN{} derivative $\frac{\di \tau(\para,N_{\PPP(\nu_H,\nu_V)})}{\di r}$ and, moreover, it is explicit:
\begin{lemma} \label{lem:density}
  Take a ballistic model with parameter $\para$ and any counting measure $N_{\PPP(\nu_H,\nu_V)}$.\par
  Then, the law $\tau(\para,N_{\PPP(\nu_H,\nu_V)})$ is absolutely continuous according to the measure $r$ and its \RN{} derivative is
  \begin{align}
    \frac{\di \tau(\para,N_{\PPP{(\nu_H,\nu_V)}})}{\di r}(U) & = \sum_{K \in \mathcal{K}} \ind{U \in K} \, \nu_V^{\tVE(U)} \nu_H^{\tHE(U)} e^{-2 \nu_V a} e^{-2 \nu_H b} e^{-(\lambda_V + \tau_V) L_V(U)} e^{-(\lambda_H + \tau_H) L_H(U)} e^{-4 \lambda_0 ab} \nonumber \\
    & \qquad \lambda_0^{\tOB(U)} \lambda_H^{\tVB(U)} \lambda_V^{\tHB(U)} \tau_H^{\tVT(U)} \tau_V^{\tHT(U)} p_V^{\tHA(U)} p_H^{\tVA(U)} p_0^{\tOA(U)} (1-p_V-p_H-p_0)^{\tCC(U)} \label{eq:density}
  \end{align}
  where $L_V(U)$ and $L_H(U)$ are the total lengths of the vertical and horizontal segments of $U$ that are
  \begin{equation}
    L_V(U) = \sum_{[(x,y),(x,y')] \in U \cap \set_V} (y'-y) \text{ and } L_H(U) = \sum_{[(x,y),(x',y)] \in U \cap \set_H} (x'-x).
  \end{equation}
\end{lemma}

\begin{remark}
  Previous lemma can be improved considering a larger class of initial conditions. But, as it requires more formalisms about random point processes on $\RR$ to describe it well, we prefer to state it only for the $N_{\PPP(\nu_H,\nu_V)}$ initial conditions as it is sufficient for this article.
\end{remark}

\begin{proof}
  The proof is similar to the one of~\cite[Lemma~5.2]{BCES23}.\par
  The Equation~\eqref{eq:density} is just a rearrangement of a product of terms where each of them represented the probability of a local event and where all the events are independent. Let us analyse each term separately.\par
  The first term $\nu_V^{\tVE(U)} \nu_H^{\tHE(U)} e^{-2 \nu_V a} e^{-2 \nu_H b}$ is the probability under $N_{\PPP(\nu_H,\nu_V)}$ that we get exactly $\VE(U)$ vertical ordered particles and $\HE(U)$ horizontal ordered particles.\par
  The term $\lambda_0^{\tOB(U)} e^{-4 \lambda_0 ab}$ is the probability for a PPP with intensity $\lambda_0$ to get $\OB(U)$ points on the rectangle $[-a,a] \times [-b,b]$.\par 
  Here, we should ask why $\frac{1}{n!}$ or $\frac{1}{m!}$ or $\frac{1}{\OB(U)!}$ terms do not appear. This is due to the fact that in our reference measure $r$, the ordinates and abscissas are ordered. And so, when the integration is made according to the reference measure $r$, the factorials appear back.\par
  \medskip
  Now, for a vertical particle, two events can occur. Either it splits or turn at rate $\tau_V + \lambda_V$ and, when this happens, with probability $\frac{\tau_V}{\tau_V + \lambda_V}$, it is a turn and, with probability $\frac{\lambda_V}{\tau_V + \lambda_V}$, it is a split. This explains the two terms $\lambda_V^{\tHB(U)}$ and $\tau_V^{\tHT(U)}$, and, similarly considering horizontal particles, the two terms $\lambda_H^{\tVB(U)}$ and $\tau_H^{\tVT(U)}$. Either it encounters a horizontal particle, in that case, four issues can happen with probabilities $p_V$, $p_H$, $p_0$ and $1-p_V-p_H-p_0$. That induces the term $p_V^{\tHA(U)} p_H^{\tVA(U)} p_0^{\tOA(U)} (1-p_V-p_H-p_0)^{\tCC(U)}$.\par
  Moreover, between two splits, turns or encounters, a term $e^{-(\lambda_V + \tau_V) d}$ appears, where $d$ is the distance between the two events. This term is the probability for the vertical particle to do not split or turn during a distance $d$. All these contributions, multiplied together, give the term $e^{-(\lambda_V + \tau_V) L_V(U)}$ and its horizontal counterpart $e^{-(\lambda_H + \tau_H) L_H(U)}$.
\end{proof}

\subsection{Proof of Theorem~\ref{thm:pi} and Corollary~\ref{cor:pi}} \label{sec:proof-pi}
We begin this section with the proof of Theorem~\ref{thm:pi} and then we do the one of Corollary~\ref{cor:pi}.
\begin{proof}[Proof of Theorem~\ref{thm:pi}]
Finally, to prove Theorem~\ref{thm:pi}, we just need to check that, for any $a,b > 0$, the density of $\tau(\para,N_{\PPP(\nu_H,\nu_V)})$ that is, for any $U \in \mathcal{U}$,
\begin{align} \label{eq:densU}
  & \sum_{K \in \mathcal{K}} \ind{U \in K}\, \nu_V^{\tVE(U)} \nu_H^{\tHE(U)} e^{-2 \nu_V a} e^{-2 \nu_H b} e^{-(\lambda_V + \tau_V) L_V(U)} e^{-(\lambda_H + \tau_H) L_H(U)} e^{-4 \lambda_0 ab} \nonumber \\
  & \qquad \lambda_0^{\tOB(U)} \lambda_H^{\tVB(U)} \lambda_V^{\tHB(U)} \tau_H^{\tVT(U)} \tau_V^{\tHT(U)} p_V^{\tHA(U)} p_H^{\tVA(U)} p_0^{\tOA(U)} (1-p_V-p_H-p_0)^{\tCC(U)}
\end{align}
is equal to the one of $\tau(\tilde{\para},N_{\PPP(\nu_H,\nu_V)})$ taken in $s_\pi(U) = \tilde{U}$ that is
\begin{align} \label{eq:densTU}
  & \sum_{K \in \mathcal{K}} \ind{s_\pi(U) \in K}\, \nu_V^{\tVE(\tilde{U})} \nu_H^{\tHE(\tilde{U})} e^{-2 \nu_V a} e^{-2 \nu_H b} e^{-(\tilde{\lambda}_V + \tilde{\tau}_V) L_V(\tilde{U})} e^{-(\tilde{\lambda}_H + \tilde{\tau}_H) L_H(\tilde{U})} e^{-4 \tilde{\lambda}_0 ab} \nonumber \\
  & \qquad \tilde{\lambda}_0^{\tOB(\tilde{U})} \tilde{\lambda}_H^{\tVB(\tilde{U})} \tilde{\lambda}_V^{\tHB(\tilde{U})} \tilde{\tau}_H^{\tVT(\tilde{U})} \tilde{\tau}_V^{\tHT(\tilde{U})} \tilde{p}_V^{\tHA(\tilde{U})} \tilde{p}_H^{\tVA(\tilde{U})} \tilde{p}_0^{\tOA(\tilde{U})} (1-\tilde{p}_V-\tilde{p}_H-\tilde{p}_0)^{\tCC(\tilde{U})}
\end{align}
under the five conditions given in Theorem~\ref{thm:pi}.\par

The rotation $s_\pi$ transforms points of $\tilde{U}$ into points of $U$ with a correspondence between types, see Table~\ref{tab:type} for the correspondence. For example, this implies that $\VT(\tilde{U}) = \HT(U)$.
\begin{table}
  \begin{center}
    \begin{tabular}{|c||c|c|c|c|c|c|c|c|c|c|c|}
      \hline
      $s_{\pi}(U)$ & \VE & \HE & \OB & \VB & \HB & \VT & \HT & \HA & \VA & \OA & \CC \\
      \hline
      $U$         & \VS & \HS & \OA & \VA & \HA & \HT & \VT & \HB & \VB & \OB & \CC \\
      \hline
    \end{tabular}  
    \caption{Correspondence between type of points in $s_\pi(U)$ and in $U$. For instance, a point that is a creation of a vertical particle by a horizontal one (\VB) in $s_\pi(U)$ corresponds to a kill of a vertical particle by a horizontal one that survives (\VA) in $U$.}%
    \label{tab:type}
  \end{center}
\end{table}

The skeleton relation is compatible with the rotation $s_\pi$, in the sense that $K(s_\pi(U)) = s_\pi(K(U))$, and so $\ind{U \in K} = \ind{s_\pi(U) \in s_\pi(K)}$.\par

The rotation $s_\pi$ transforms a vertical segment $[(x,y),(x',y)]$ of length $x'-x$ into the vertical one $[(-x',-y),(-x,-y)]$ of the same length. Hence, the total length of vertical segments $L_V$ is the same in $U$ and in $s_\pi(U) = \tilde{U}$, similarly for the horizontal ones:
\begin{lemma} \label{lem:lenght}
  $L_V(\tilde{U})= L_V(U)$ and $L_H(\tilde{U}) = L_H(U)$.
\end{lemma}

By Table~\ref{tab:type}, Lemma~\ref{lem:lenght} and the five conditions given in Theorem~\ref{thm:pi}, the density given by Equation~\eqref{eq:densTU} rewrites 
\begin{align}
  & \nu_V^{\tVS(U)} \nu_H^{\tHS(U)} e^{-2 \nu_V a} e^{-2 \nu_H b} e^{-(\lambda_V + \tau_V) L_V(U)} e^{-(\lambda_H + \tau_H) L_H(U)} e^{-4 \lambda_0 ab} (\nu_V \nu_H p_0)^{\tOA(U)} (\nu_V p_H)^{\tVA(U)}  \nonumber \\
  & \qquad (\nu_H p_V)^{\tHA(U)} \left(\frac{\nu_V \tau_V}{\nu_H}\right)^{\tHT(U)} \left(\frac{\nu_H \tau_H}{\nu_V}\right)^{\tVT(U)} \left(\frac{\lambda_V}{\nu_H} \right)^{\tHB(U)} \left(\frac{\lambda_H}{\nu_V} \right)^{\tVB(U)} \left(\frac{\lambda_0}{\nu_V \nu_H}\right)^{\tOB(U)} (1-p_V-p_H-p_0)^{\tCC(U)}.
\end{align}

Comparing it with the density given by Equation~\eqref{eq:densU}, we need to check that the following equation holds:
\begin{equation}
  \nu_V^{\tVE} \nu_H^{\tHE} = \nu_V^{\tVS+\tOA+\tVA+\tHT-\tVT-\tVB-\tOB} \nu_H^{\tHS+\tOA+\tHA-\tHT+\tVT-\tHB-\tOB} 
\end{equation}
we dropped $(U)$ here to shorten the equation. That is equivalent to
\begin{equation}
  \nu_V^{\tVE+\tVT+\tVB+\tOB} \nu_H^{\tHE+\tHT+\tHB+\tOB} = \nu_V^{\tVS+\tOA+\tVA+\tHT} \nu_H^{\tHS+\tOA+\tHA+\tVT}. 
\end{equation}

Finally, to end the proof, we use the following lemma:
\begin{lemma} \label{lem:sepoints}
  For any configuration $U \in \mathcal{U}$ (or any skeleton $K \in \mathcal{K}$),
  \begin{align*}
    & \HE + \HT + \HB + \OB = \HS + \OA + \HA + \VT = m \text{ and}\\
    & \VE + \VT + \VB + \OB = \VS + \OA + \VA + \HT = n.
  \end{align*}
\end{lemma}

\begin{proof}
  Every segment has a beginning and an end. The number $\HE + \HT + \HB + \OB$ is the number of beginnings of a horizontal segment and the number $\HS + \OA + \HA + \VT$ is the number of endings.
\end{proof}

This lemma ends the first part of the proof. Hence, for any $a,b > 0$, we prove that
\begin{displaymath}
  \mathcal{L}_{[-a,\infty) \times [-b,\infty)}(\para,N_{\PPP(\nu_H,\nu_V)})|_{[-a,a]\times[-b,b]} \eqd s_\pi\left(\mathcal{L}_{[-a,\infty) \times [-b,\infty)}(\tilde{\para},N_{\PPP(\nu_H,\nu_V)})|_{[-a,a]\times[-b,b]} \right).
\end{displaymath}
By the translation invariance of the model, this is true for any rectangle considering the rotation by an angle $\pi$ and whose centre of rotation is the centre of the rectangle, not just the rectangle $[-a,a] \times [-b,b]$. 

The last point we need to prove now is that $N_{\PPP(\nu_H,\nu_V)}$ is stationary for the \Bullet model with parameter $\para$. Let's suppose that the $x$-and $y$-axes are distributed according to $N_{\PPP(\nu_H,\nu_V)}$ and from this initial condition, we apply a \Bullet model with parameter $\para$. We want to prove that, for almost any $a,b,c,d > 0$,
\begin{displaymath}
  \mathcal{L}_{[0,\infty)^2}(\para,N_{\PPP(\nu_H,\nu_V)})|_{[a,a+c]\times[b,b+d]} \eqd \mathcal{L}_{[a,\infty) \times [b,\infty)}(\para,N_{\PPP(\nu_H,\nu_V)})|_{[a,a+c]\times[b,b+d]}.
\end{displaymath}
For that, consider the rectangle $[0,a+c] \times [0,b+d]$, by the first result of the proof, the space-time diagram inside this rectangle is the $s_\pi$-reverse one of the \Bullet model with parameter $\tilde{\para}$ with initial condition $N_{\PPP(\nu_H,\nu_V)}$, and so
\begin{displaymath}
  \mathcal{L}_{[0,\infty)^2}(\para,N_{\PPP(\nu_H,\nu_V)})|_{[a,a+c]\times[b,b+d]} \eqd s_\pi \left( \mathcal{L}_{[-a-c,\infty) \times [-b-d,\infty)}(\tilde{\para},N_{\PPP(\nu_H,\nu_V)})|_{[-a-c,-a]\times[-b-d,-b]} \right).
\end{displaymath}

Now, consider the rectangle $[a,a+c] \times [b,b+d]$, once again, we get that the space-time diagram inside it is the $s_\pi$-reverse one of the \Bullet model with parameter $\tilde{\para}$ with initial condition $N_{\PPP(\nu_H,\nu_V)}$, that is exactly the one of a \Bullet model with parameter $\para$ and initial condition $N_{\PPP(\nu_H,\nu_V)}$, that is
\begin{displaymath}
s_\pi\left(\mathcal{L}_{[-a-c,\infty) \times [-b-d,\infty)}(\tilde{\para},N_{\PPP(\nu_H,\nu_V)})|_{[-a-c,-a]\times[-b-d,-b]}\right) \eqd \mathcal{L}_{[a,\infty) \times [b,\infty)}(\para,N_{\PPP(\nu_H,\nu_V)})|_{[a,a+c]\times[b,b+d]}.
\end{displaymath}

This is illustrated on Figure~\ref{fig:impi}.
\end{proof}

\begin{figure}
  \begin{center}
    \begin{tabular}{ccc}
      \includegraphics{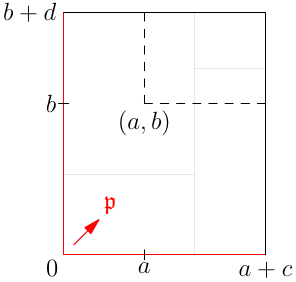} & \includegraphics{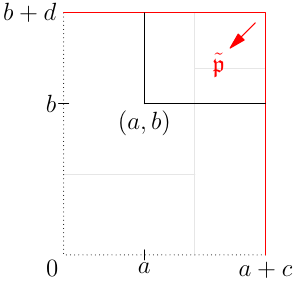} & \includegraphics{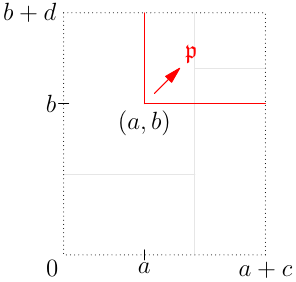}
    \end{tabular}
  \end{center}
  \caption{On the left, the initial situation: a realisation (in grey) of the \Bullet model with parameter $\para$ and initial condition $N_{\PPP(\nu_H,\nu_V)}$. On the centre, by the first part of the proof applied to the rectangle $[0,a+c] \times [0,b+d]$: the rotation by angle $\pi$ of its realisation is the one of a \Bullet model with parameter $\tilde{\para}$ under initial condition $N_{\PPP(\nu_H,\nu_V)}$. On the right, once again by the first part of the proof applied to the rectangle $[a,a+c] \times [b,b+d]$: the realisation restrained to the smaller rectangle is the one of a \Bullet model $\para$ with initial condition $N_{\PPP(\nu_H,\nu_V)}$.}%
  \label{fig:impi}
\end{figure}

\begin{proof}[Proof of Corollary~\ref{cor:pi}]
  Consider a \Bullet model with parameter $\para$ such that $A,B_V,B_H$, defined in Equations~\eqref{eq:A},~\eqref{eq:BV} and~\eqref{eq:BH}, are all not zero, and such that $B_H B_V \lambda_ 0 = A^2 p_0 \tau_V$. Then, we take $\nu_H = A/B_H$, $\nu_V = A/B_V$ and $\tilde{\para}$ defined by:
  \begin{itemize}
  \item $\tilde{\lambda}_0 = \lambda_0$, $\tilde{\lambda}_V = \frac{A}{B_H} p_V$ and $\tilde{\lambda}_H = \frac{A}{B_V} p_H$,
  \item $\tilde{\tau}_V = \frac{B_V}{B_H} \tau_H$ and $\tilde{\tau}_H = \frac{B_H}{B_V} \tau_V$,
  \item $\tilde{p}_V = \frac{B_H}{A} \lambda_V$, $\tilde{p}_H = \frac{B_V}{A} \lambda_H$ and $\tilde{p}_0 = p_0$.
  \end{itemize}

  Let us check now that the five conditions of Theorem~\ref{thm:pi} are satisfied.
  \begin{enumerate}
  \item $\tilde{\lambda}_H + \tilde{\tau}_H = \lambda_H + \tau_H$?
    \begin{align*}
      \tilde{\lambda}_H + \tilde{\tau}_H & = \frac{A}{B_V} p_H + \frac{B_H}{B_V} \tau_V\\
                                         & = \frac{A p_H + B_H \tau_V}{B_V}\\
                                         & = \frac{(\lambda_H + \tau_H) (\lambda_V + \tau_V) p_H - \tau_V \tau_H p_H + (p_H+p_V)(\lambda_H + \tau_H) \tau_V - p_H \lambda_H \tau_V}{B_V} \\
                                         & = \frac{(\lambda_H + \tau_H)[(\lambda_V + \tau_V) p_H + (p_H+p_V) \tau_V - \tau_V p_H]}{B_V} = (\lambda_H+\tau_H) \frac{B_V}{B_V}.
    \end{align*}
    Similarly, $\tilde{\lambda}_V + \tilde{\tau}_V = \lambda_V + \tau_V$.
  \item $\tilde{\lambda}_0 = \lambda_0 = \frac{A^2}{B_H B_V} p_0 = \nu_H \nu_V p_0$, $\tilde{\lambda}_V = \frac{A}{B_H} p_V = \nu_H p_V$ and, similarly, $\tilde{\lambda}_H = \frac{A}{B_V} p_H = \nu_V p_H$.
  \item $\nu_V \tilde{\tau}_V = \nu_V \frac{B_V}{B_H} \tau_H = \nu_V \frac{\nu_H}{\nu_V} \tau_H$ and, similarly, $\nu_H \tilde{\tau}_H = \nu_V \tau_V$.
  \item $\nu_H \tilde{p}_V  = \frac{A}{B_H} \frac{B_H}{A} \lambda_V$ and, similarly, $\nu_V \tilde{p}_H = \lambda_H$.
  \item $\tilde{p}_V + \tilde{p}_H + \tilde{p}_0 = p_V + p_H + p_0$?
    \begin{align*}
      \tilde{p}_H + \tilde{p}_V + \tilde{p}_0 & = \frac{B_V}{A} \lambda_H + \frac{B_H}{A} \lambda_V + p_0\\
                                              & = \frac{B_V \lambda_H + B_H \lambda_V}{A} + p_0\\
                                              & = \frac{(p_H+p_V)(\lambda_V + \tau_V) \lambda_H - p_V \lambda_V \lambda_H + (p_H+p_V)(\lambda_H + \tau_H) \lambda_V - p_H \lambda_V \lambda_H }{A} + p_0\\
                                              & = \frac{(p_H+p_V)[(\lambda_V + \tau_V) \lambda_H + (\lambda_H + \tau_H) \lambda_V - \lambda_V \lambda_H]}{A} + p_0\\
                                              & = (p_H+p_V) \frac{A}{A} + p_0.
    \end{align*}
  \end{enumerate}
  Hence, the five conditions of Theorem~\ref{thm:pi} are satisfied. That's end the proof of the corollary.
\end{proof}

\subsection{Proof of Theorem~\ref{thm:pi2} and Corollary~\ref{cor:pi2}} \label{sec:proof-pi2}
We begin this section with the proof of Theorem~\ref{thm:pi2} and then we do the one of Corollary~\ref{cor:pi2}.
\begin{proof}[Proof of Theorem~\ref{thm:pi2}]
The proof of Theorem~\ref{thm:pi2} is very similar to the one of Theorem~\ref{thm:pi}. We need to check that, for any $a,b > 0$, for any $\nu_H \geq 0$, the density of $\tau(\para,N_{\PPP(\nu_H,\nu_V)})$ that is, for any $U \in \mathcal{U}$,
\begin{align} \label{eq:densU2}
  & \sum_{K \in \mathcal{K}} \ind{U \in K}\, \nu_V^{\tVE(U)} \nu_H^{\tHE(U)} e^{-2 \nu_V a} e^{-2 \nu_H b} e^{-(\lambda_V + \tau_V) L_V(U)} e^{-(\lambda_H + \tau_H) L_H(U)} e^{-4 \lambda_0 ab} \nonumber \\
  & \qquad \lambda_0^{\tOB(U)} \lambda_H^{\tVB(U)} \lambda_V^{\tHB(U)} \tau_H^{\tVT(U)} \tau_V^{\tHT(U)} p_V^{\tHA(U)} p_H^{\tVA(U)} p_0^{\tOA(U)} (1-p_V-p_H-p_0)^{\tCC(U)}
\end{align}
is equal to the one of $\tau(\tilde{\para},N_{\PPP(\nu_V,\nu_H)})$ taken in $s_{\pi/2}(U) = \tilde{U}$ that is
\begin{align} \label{eq:densTU2}
  & \sum_{K \in \mathcal{K}} \ind{s_{\pi/2}(U) \in K}\, \nu_H^{\tVE(\tilde{U})} \nu_V^{\tHE(\tilde{U})} e^{-2 \nu_H b} e^{-2 \nu_V a} e^{-(\tilde{\lambda}_V + \tilde{\tau}_V) L_V(\tilde{U})} e^{-(\tilde{\lambda}_H + \tilde{\tau}_H) L_H(\tilde{U})} e^{-4 \tilde{\lambda}_0 ba} \nonumber \\
  & \qquad \tilde{\lambda}_0^{\tOB(\tilde{U})} \tilde{\lambda}_H^{\tVB(\tilde{U})} \tilde{\lambda}_V^{\tHB(\tilde{U})} \tilde{\tau}_H^{\tVT(\tilde{U})} \tilde{\tau}_V^{\tHT(\tilde{U})} \tilde{p}_V^{\tHA(\tilde{U})} \tilde{p}_H^{\tVA(\tilde{U})} \tilde{p}_0^{\tOA(\tilde{U})} (1-\tilde{p}_V-\tilde{p}_H-\tilde{p}_0)^{\tCC(\tilde{U})}.
\end{align}

The rotation $s_{\pi/2}$ transforms points of $\tilde{U}$ into points of $U$ with a correspondence between types, see Table~\ref{tab:type2} for the correspondence. For example, this implies that $\HT(\tilde{U}) = \OA(U)$.
\begin{table}
  \begin{center}
    \begin{tabular}{|c||c|c|c|c|c|c|c|c|c|c|c|}
      \hline
      $s_{\pi/2}(U)$ & \VE & \HE & \OB & \VB & \HB & \VT & \HT & \HA & \VA & \OA & \CC \\
      \hline
      $U$           & \HE & \VS & \HT & \HB & \VA & \OB & \OA & \VB & \HA & \VT & \CC \\
      \hline
    \end{tabular}  
    \caption{Correspondence between type of points in $s_{\pi/2}(U)$ and in $U$. For instance, a point that is a creation of a vertical particle by a horizontal one (\VB) in $s_{\pi/2}(U)$ corresponds to a creation of a horizontal particle by a vertical one (\HB) in $U$.}%
    \label{tab:type2}
  \end{center}
\end{table}

The skeleton relation is compatible with the rotation $s_{\pi/2}$, in the sense that $K(s_{\pi/2}(U)) = s_{\pi/2}(K(U))$, and so $\ind{U \in K} = \ind{s_{\pi/2}(U) \in s_{\pi/2}(K)}$.\par

The rotation $s_{\pi/2}$ transforms a vertical segment of a certain length into a horizontal one of the same length. Hence,
\begin{lemma} \label{lem:length2}
  $L_V(\tilde{U})= L_H(U)$ and $L_H(\tilde{U}) = L_V(U)$.
\end{lemma}

All of these points and the five conditions given in Theorem~\ref{thm:pi2} permit to rewrite Equation~\eqref{eq:densTU2} as
\begin{align*}
  & \nu_H^{\tHE} \nu_V^{\tVS} e^{-2 \nu_H b} e^{-2 \nu_V a} e^{-(\lambda_H + \tau_H) L_H} e^{-(\lambda_V + \tau_V) L_V} e^{-4 \lambda_0 ba} \nonumber \\
  & \qquad (\nu_V \tau_V)^{\tHT} \lambda_V^{\tHB}(\nu_V p_H)^{\tVA} \left(\frac{\lambda_0}{\nu_V}\right)^{\tOB} (\nu_V p_0)^{\tOA} \left(\frac{\lambda_H}{\nu_V}\right)^{\tVB} p_V^{\tHA} \left(\frac{\tau_H}{\nu_V} \right)^{\tVT} (1-p_V-p_H-p_0)^{\tCC}.
\end{align*}
We omitted $(U)$ in the previous equation.

After simplification with Equation~\eqref{eq:densU2}, we have to check if the following equation holds:
\begin{displaymath}
  \nu_V^{\tVE} \nu_H^{\tHE} = \nu_H^{\tHE} \nu_V^{\tVS + \tHT +\tVA - \tOB + \tOA - \tVB - \tVT}
\end{displaymath}
that is equivalent to $\nu_V^{\tVE + \tOB + \tVB + \tVT } = \nu_V^{\tVS + \tHT +\tVA + \tOA}$ that is true by Lemma~\ref{lem:sepoints}.\par
\medskip
The last point we need to prove now is that if $N_{\PPP(\nu_H,\nu_V)}$ is stationary for the \Bullet model with parameter $\para$, then $N_{\PPP(\nu_V,\nu_H)}$ is stationary for the \Bullet model with parameter $\tilde{\para}$. As in the end of the proof of Theorem~\ref{thm:pi}, we just need to consider the good rectangles to do it.\par
\bigskip
Let's suppose that the $x$-and $y$-axes are distributed according to $N_{\PPP(\nu_V,\nu_H)}$ and from this initial condition, we apply a \Bullet model with parameter $\tilde{\para}$. We want to prove that, for almost any $a,b,c,d > 0$,
\begin{displaymath}
  \mathcal{L}_{[0,\infty)^2}(\tilde{\para},N_{\PPP(\nu_V,\nu_H)})|_{[a,a+c]\times[b,b+d]} \eqd \mathcal{L}_{[a,\infty) \times [b,\infty)}(\tilde{\para},N_{\PPP(\nu_V,\nu_H)})|_{[a,a+c]\times[b,b+d]}.
\end{displaymath}
For that, consider first, the rectangle $[0,a+c] \times [0,b+d]$, by the result already proved,
\begin{displaymath}
  \mathcal{L}_{[0,\infty)^2}(\tilde{\para},N_{\PPP(\nu_V,\nu_H)})|_{[a,a+c]\times[b,b+d]} \eqd s_{\pi/2}\left( \mathcal{L}_{[0,\infty) \times [-a-c,\infty)}(\para,N_{\PPP(\nu_H,\nu_V)})|_{[b,b+d]\times[-a-c,-a]} \right).
\end{displaymath}
By hypothesis, the law $N_{\PPP(\nu_H,\nu_V)}$ is stationary for the \Bullet model with parameter $\para$. Applying it to the rectangle $[0,a+c]\times[b,b+d]$, we get
\begin{displaymath}
  \mathcal{L}_{[0,\infty) \times [-a-c,\infty)}(\para,N_{\PPP(\nu_H,\nu_V)})|_{[b,b+d]\times[-a-c,-a]} \eqd \mathcal{L}_{[b,\infty) \times [-a-c,\infty)}(\para,N_{\PPP(\nu_H,\nu_V)})|_{[b,b+d]\times[-a-c,-a]}.
\end{displaymath}
Finally, we apply once again the result already proved to the rectangle $[a,a+c] \times [b,b+d]$ to get
\begin{displaymath}
  \mathcal{L}_{[a,\infty) \times [b,\infty)}(\tilde{\para},N_{\PPP(\nu_V,\nu_H)})|_{[a,a+c]\times[b,b+d]} \eqd s_{\pi/2}\left( \mathcal{L}_{[b,\infty) \times [-a-c,\infty)}(\para,N_{\PPP(\nu_H,\nu_V)})|_{[b,b+d]\times[-a-c,-a]} \right).
\end{displaymath}
This is illustrated on Figure~\ref{fig:impi2}.
\end{proof}

\begin{figure}
  \begin{center}
    \begin{tabular}{cccc}
      \includegraphics[width=0.22 \textwidth]{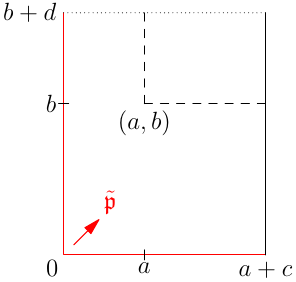} & \includegraphics[width=0.22 \textwidth]{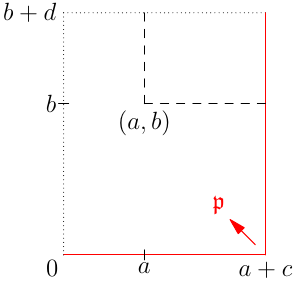} & \includegraphics[width=0.22 \textwidth]{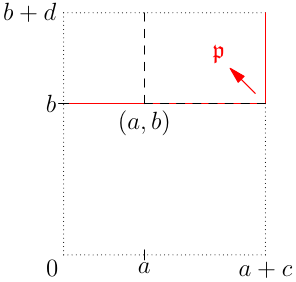} & \includegraphics[width=0.22 \textwidth]{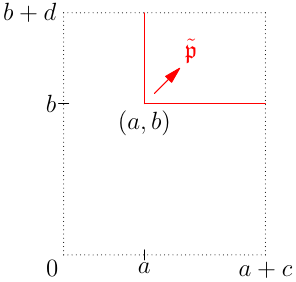}
    \end{tabular}
  \end{center}
  \caption{The successive rectangle considered to prove the stationarity of $(\para,N_{\PPP(\nu_V,\nu_H)})$.}%
  \label{fig:impi2}
\end{figure}

\begin{proof}[Proof of Corollary~\ref{cor:pi2}]
  Consider a \Bullet model with parameter $\para$ such that $A,B_V,B_H$, defined in Equations~\eqref{eq:A},~\eqref{eq:BV} and~\eqref{eq:BH}, are all not zero, and such that $B_V \lambda_ 0 = A \tau_V$ and $A p_0 = B_H \tau_V$. Then, we take $\nu_H = A/B_H$, $\nu_V = A/B_V$ and $\tilde{\para}$ defined by:
  \begin{itemize}
  \item $\tilde{\lambda}_0 = \frac{A}{B_V} \tau_V = \lambda_0$, $\tilde{\lambda}_V = \frac{A}{B_V} p_H$ and $\tilde{\lambda}_H = \lambda_V$,
  \item $\tilde{\tau}_V = \frac{B_H}{B_V} \tau_V = \frac{A}{B_V} p_0$ and $\tilde{\tau}_H = \tau_V = \frac{B_V}{A} \lambda_0$,
  \item $\tilde{p}_V = \frac{B_V}{A} \lambda_H$, $\tilde{p}_H = p_V$ and $\tilde{p}_0 = \frac{B_V}{A} \tau_H$.
  \end{itemize}

  Let us check now that the five conditions of Theorem~\ref{thm:pi2} are satisfied.
  \begin{enumerate}
  \item $\tilde{\lambda}_H + \tilde{\tau}_H = \lambda_V + \tau_V$ is trivial. Now, $\tilde{\lambda}_V + \tilde{\tau}_V = \lambda_H + \tau_H$?
    \begin{align*}
      \tilde{\lambda}_V + \tilde{\tau}_V & = \frac{A}{B_V} p_H + \frac{B_H}{B_V} \tau_V = \frac{A p_H + B_H \tau_V}{B_V}\\
                                         & = \frac{(\lambda_H + \tau_H) (\lambda_V + \tau_V) p_H - \tau_V \tau_H p_H + (p_H+p_V)(\lambda_H + \tau_H) \tau_V - p_H \lambda_H \tau_V}{B_V} \\
                                         & = \frac{(\lambda_H + \tau_H)[(\lambda_V + \tau_V) p_H + (p_H+p_V) \tau_V - \tau_V p_H]}{B_V} = (\lambda_H+\tau_H) \frac{B_V}{B_V}.
    \end{align*}
  \item $\tilde{\lambda}_0 = \lambda_0 = \frac{A}{B_V} \tau_V = \nu_V \tau_V$, $\tilde{\lambda}_V = \frac{A}{B_V} p_H = \nu_V p_H$.
  \item $\tilde{\tau}_V = \frac{A}{B_V} p_0 = \nu_V p_0$, and $\nu_V \tilde{\tau}_H = \frac{A}{B_V} \frac{B_V}{A} \lambda_0$.
  \item $\nu_V \tilde{p}_V = \frac{A}{B_V} \frac{B_V}{A} \lambda_H$, and $\nu_V \tilde{p}_0 = \frac{A}{B_V} \frac{B_V}{A} \tau_H$.
  \item $\tilde{p}_V + \tilde{p}_H + \tilde{p}_0 = p_V + p_H + p_0$?
    \begin{align*}
      \tilde{p}_V + \tilde{p}_0 + \tilde{p}_H & = \frac{B_V}{A} \lambda_H + \frac{B_V}{A} \tau_H + p_V = \frac{B_V}{A} (\lambda_H + \tau_H) + p_V \\
                                              & = \frac{B_V}{A} (\tilde{\lambda}_V + \tilde{\tau}_V) + p_V = \frac{B_V}{A} \tilde{\lambda}_V  + \frac{B_V}{A} \tilde{\tau}_V + p_V\\
                                              & = p_H + p_0 + p_V.                               
    \end{align*}
  \end{enumerate}
  Hence, the five conditions of Theorem~\ref{thm:pi2} are satisfied. That's end the proof of the corollary.
\end{proof}

\section{Thanks}
The author would like to thank Anne Briquet for her presentation about the colliding \Bullet model with creations many years ago, which inspires this paper. It would also like to thank Arvind Singh for pointing him to the article about Hammersley's tree and Matteo d'Achille for a short discussion about different kinds of \Loop models.

\bibliographystyle{alpha}
\bibliography{aaa}

\appendix

\section{Simulations code} \label{ann:code}
The realisations of Figure~\ref{fig:examples} are obtained via simulations coded in the Asymptote language~\cite{Asymptote}.

\lstdefinestyle{myStyle}{
    belowcaptionskip=1\baselineskip,
    breaklines=true,
    frame=none,
    numbers=none, 
    basicstyle=\footnotesize\ttfamily,
    keywordstyle=\bfseries\color{green!40!black},
    commentstyle=\itshape\color{red!40!black},
    identifierstyle=\color{blue},
    backgroundcolor=\color{gray!10!white},
}

\begin{lstlisting}[language=C++,style=myStyle]
settings.outformat = "pdf" ;

size(20cm,20cm) ; //size of the final drawing

srand(3) ; //seed

//size of the rectangle
real haut = 5 ; //height of the rectangle
real larg = 5 ; //width of the rectangle
draw((0,0)--(0,haut)--(larg,haut)--(larg,0)--cycle,white) ;

//parameters ;
real lambda0 = 1.2 ; //spontaneous creation rate
real lambdaV = 0.8 ; //vertical split rate
real lambdaH = 0.7 ; //horizontal split rate
real tauV = 2.3 ; //vertical turn rate
real tauH = 1.7 ; //horizontal turn rate
real p0 = 0.2 ; //annihilation probability
real pV = 0.3 ; //vertical coalescensce probability
real pH = 0.4 ; //horizontal coalescensce probability
real pC = 1-p0-pV-pH ;

//parameter of the initial conditions
real nuV = 0.4 ; //initial vertical rate
real nuH = 1.1 ; //initial horizontal rate

//comparison of reals
bool pp(real x1, real x2) {
  return x1 < x2 ;
}

//comparison of pairs (x,y) by their ordinates
bool cmp(pair p1, pair p2) {
  return p1.y < p2.y;
}

//return an exponential random variable with parameter lambda (mean = 1/lambda)
real random_exponential(real lambda) {
  return -1/lambda * log(1 - unitrand());
}

//uniform law on {1,2,...,n}
int random_uniform_int(int n){
  int res = 0 ;
  real aux = n*unitrand() ;
  for(int i = 0 ; i< aux ; ++i){
    res = res + 1 ;
  } ;
  return res ;
}
      
//return a Poisson random variable with parameter lambda
int random_poisson(real lambda){
  real sum = 0 ;
  int j = 0 ;
  for(int i ; sum < lambda ; ++i){
    sum = sum + random_exponential(1) ;
    j=i ;
  } ;
  return j ;
}

//return np sorted random reals uniformly distributed on [deb,fin]
real[] random_ppp(int np,real deb, real fin){
  real[] aux ;
  for(int i=0 ; i< np ; ++i){
    aux[i] = deb + (fin-deb)*unitrand() ;
  } ;
  real[] res = sort(aux,pp);
  return res ;
}

//return np sorted (by ordinates) pairs uniformly distributed on the rectangle [0,l] x [0,h] 
pair[] random_ppp_2(int np, real l, real h){
  pair[] aux ;
  for(int i = 0 ; i<np ; ++i){
    aux[i] = (l*unitrand(),h*unitrand() ) ;
  } ;
  pair[] res = sort(aux,cmp);
  return res ;
}

int n0 = random_poisson(lambda0*haut*larg) ; //number of creation points
int nH = random_poisson(nuH*haut) ; //number of horizontal bullets in the initial condition
int nV = random_poisson(nuV*larg) ; //number of vertical bullets in the initial condition

pen hori = black ; //color of horizontal lines/bullets
pen vert = black ; //color of vertical lines/bullets

//(sorted) PPP of the vertical bullets
real[] init = random_ppp(nV,0,larg) ;

//PPP of creation points
pair[] ppp = random_ppp_2(n0,haut,larg) ;
//we add it the PPP of horizontal bullets
for(int i = n0 ; i < n0 + nH ; ++i){
  ppp[i] = (0,haut*unitrand()) ;
} ;
//we sort them by ordinates
pair[] creas = sort(ppp,cmp) ;

//config is current ordinate in index 0 and then the ordered current positions (in abscissa) of vertical bullets
real[] config ; 
config[0] = 0 ; // ordinate
for(int i=1 ; i <= nV ; ++i){
  config[i] = init[i-1] ;
} ;

//return the index (< n0+nH) of the first point in creas (PPP of both creation points and horizontal entry bullets) above yy
//return n0+nH else
int next_crea(real yy){
  int res = 0 ;
  for(int i=0 ; i < n0+nH ; ++i){
    if(creas[i].y <= yy){
      res = res + 1 ;
    } ;
  } ;
  return res ;
} ;

//return the next ordinate at which a vertical turn or split occurs, if bigger than haut, return haut
//if there is no vertical bullet, return haut
real next_jump(real[] config){
  int size = config.length -1 ;
  if((lambdaV+tauV)*size == 0){
    return haut ; 
  } else {
    return min(haut,config[0] + (random_exponential((lambdaV+tauV)*size))) ;
  } ;
} ;

//draw vertical bullets from ordinate config[0] to ordinate time 
void draw_vert(real[] config,real time){
  int size = config.length -1 ;
  if(size > 0){
    for(int i=1 ; i <= size ; ++i){
      draw((config[i],config[0])--(config[i],time),vert) ;
    } ;
  } ;
} ;

//return the list of vertical bullets in config that survive to a horizontal bullet starting from abscissa debx and ending at finx
//vertical bullets in debx or in finx (if they exist) never survive (in this function)
real[] survive(real[] config, real debx, real finx){
  real[] res ;
  int j = 0 ;
  for(int i=1 ; i <= config.length-1 ; ++i){
    if(config[i] < debx || config[i] > finx){
      res[j] = config[i] ;
      j = j + 1 ;
    } ;
    if(config[i] > debx & config[i] < finx){
      if((pC+pH)*unitrand()<pC){
	res[j] = config[i] ;
	j=j+1 ;
      } ;
    } ;
  } ;
  return res ;
}

//return the random abscissa at which a horizontal bullet turns when its starting abscissa is deb;
//return larg if the abscissa is larger than larg
real fin_saut(real deb){
  real res = larg ;
  if(tauH > 0){
    res = min(larg,deb + random_exponential(tauH)) ;
  } ;
  return res ;
} ;

//return the label of the first vertical bullet in config that kills a horizontal bullet starting at abscissa deb and turn before finS
//if such bullet does not exist, return 0
int before_saut(real[] config, real deb, real finS){
  int res = 0 ;
  for(int i=config.length - 1 ; i > 0 ; --i){
    if((config[i] < finS) & config[i] > deb){
      if(unitrand() < p0+pV){
	res = i ;
      } ;
    } ;
  } ;
  return res ;
} ;

//concatenate the real number lists t1, t2 and t3 and sort the concatened list
real[] concatener(real[] t1, real[] t2, real[] t3){
  real[] aux ;
  int l1 = t1.length;
  int l2 = t2.length ;
  int l3 = t3.length ;
  for(int i=0 ; i < l1 ; ++i){
    aux[i] = t1[i] ;
  } ;
  for(int i=0 ; i < l2 ; ++i){
    aux[l1+i] = t2[i] ;
  } ;
  for(int i=0 ; i < l3 ; ++i){
    aux[l1+l2+i] = t3[i] ;
  } ;
  real[] res = sort(aux,pp) ;
  return res ;
} ;

//the function step take as argument a config (current time in config[0] and the sorted list of abscissas of the current vertical bullets)
//its result consists of the configuration just after one and only one of this event occurs:
//  turn of a vertical bullet (turn),
//  split of a vertical bullet (split),
//  creation of two bullets (crea),
//  entry of a horizontal bullet (entry).
//In all those case, a horizontal bullet creates some changes.
real[] step(real[] config){
  real[] res ; //the final configuration

  //time of next (in ordinate) entry or crea
  int lc = next_crea(config[0]) ;
  real tc = haut ;
  if(lc < creas.length){
    tc = creas[lc].y ; //its time
  } ;

  //time of next turn or split
  real tj = next_jump(config) ; //its time

  real xx ; //starting abscissa of the horizontal bullet that changes the configuration
  int lj = 0 ; //lj=0 if its an entry or a creation, else it is the label (in ordered abscissa) of the bullet that splits or turns

  
  if(tj <= tc){
    //next event is a turn or a split
    res[0] = tj ; //time of the final config
    lj = random_uniform_int(config.length-1) ; //the label of the bullet, that splits or turns, is uniform
    xx = config[lj] ; //abscissa of the horizontal bullet = abscissa of the bullet that splits or turns
  } else {
    //next event is a horizontal entry or a crea
    res[0] = tc ; //time of the final config
    xx = creas[lc].x ; //abscissa of the horizontal bullet = abscissa of the entry or crea
  } ;

  //we draw the current vertical bullets from initial config to final config
  draw_vert(config,res[0]) ;

  //where the horizontal bullet ends ?
  real ff = fin_saut(xx) ; //the abscissa of the turn of the horizontal bullet (or larg if it turns after or never)
  int kill = before_saut(config,xx,ff) ; //return the label of a potential vertical bullet killer
  if(kill > 0){
    ff = config[kill] ; //if such a bullet exists, then the horizontal bullet ends at the abscissa of its killer
  } ;

  //we draw the horizontal bullet
  draw((xx,res[0])--(ff,res[0]),hori) ;

  //list of vertical bullets that survive to the horizontal bullet
  real[] memory = survive(config,xx,ff) ; 

  //list of vertical bullets created by the horizontal bullet
  int nP = random_poisson(lambdaH*(ff-xx)) ;
  real[] poissonVert = random_ppp(nP,xx,ff) ;

 
  real[] debfin ;
  int j = 0 ; //auxiliary variable
  
  //adding a vertical bullet in abcsissa xx?
  if(lj > 0){ //the event is a turn or a split
    if((lambdaV+tauV)*unitrand() < lambdaV){ //proba that it is a split
      //it is a split, so we add back the vertical bullet
      debfin[j] = xx ;
      j = j + 1 ;
    }
  } else { //the event is an entry or a crea
    if(xx > 0){ //it is a crea, so we add a vertical bullet
      debfin[j] = xx ;
      j = j + 1 ;
    } ;
  } ;
  
  //adding a vertical bullet in abcsissa ff?
  if(kill > 0){ //the horizontal bullet is killed
    if((p0+pV)*unitrand() < pV){ //proba that is not a total annihilation
      //the vertical bullet survives to the collision, so we add it back
      debfin[j] = ff ;
    } ;
  } else {//the horizontal bullet turns before larg or turns/ends after larg
    if(ff < larg){ //its a turns before larg
      //the horizontal bullet becomes a vertical bullet
      debfin[j] = ff ;
    } ;
  } ;

  //we concatenate and order all the vertical bullets 
  real[] newElement = concatener(memory,poissonVert,debfin) ;
  for(int i = 1 ; i <= newElement.length ; ++i){
    res[i] = newElement[i-1] ; //we copy the vertical bullets in the new config
  } ;
  
  //res contains then res[0] the new time and res[i], i>0, the vertical bullets ordered by abscissa
  return res ;
} ;

real[] part = config ; //we start with the initial configuration called config
//until the current time is less than haut, we apply the function step
for(int i=0 ; part[0] < haut ; ++i){
  write(part[0]) ;
  write(i) ;
  part = step(part) ;
}

//we draw a dashed frame
draw((0,0)--(0,haut)--(larg,haut)--(larg,0)--cycle,white+linewidth(1pt)) ;
draw((0,0)--(0,haut)--(larg,haut)--(larg,0)--cycle,dashed+linewidth(1pt)) ;

\end{lstlisting}

\end{document}